\documentclass[12pt]{amsart}
\usepackage{times}
\usepackage[T1]{fontenc}
\usepackage{dsfont}
\usepackage{mathrsfs}
\usepackage[colorlinks]{hyperref}
\usepackage{hyperref}
\usepackage{xcolor}
\usepackage[a4paper,asymmetric]{geometry}
\usepackage{mathscinet}
\usepackage{fullpage}
\usepackage{latexsym}
\usepackage{amsthm}
\usepackage{amssymb}
\usepackage{amsfonts}
\usepackage{amsmath}
\newtheorem{theorem}{Theorem}[section]
\newtheorem{thm}[theorem]{Theorem}

\newtheorem{lem}[theorem]{Lemma}
\newtheorem{proposition}[theorem]{Proposition}

\newtheorem{corollary}[theorem]{Corollary}
\newtheorem{assumption}[theorem]{Assumption}
\theoremstyle{definition}

\newtheorem{defn}[theorem]{Definition}

\theoremstyle{remark}
\newtheorem{remark}[theorem]{Remark}
\newtheorem{rem}[theorem]{Remark}
\numberwithin{equation}{section}

 \DeclareMathAlphabet{\mathpzc}{OT1}{pzc}{m}{it}

  \newcommand{\dif}{\mathrm{d}}

 \newcommand{\E}{\mathbb{E}}
 
 \newcommand{\e}{\varepsilon}
 
 \newcommand{\p}{\partial}

 \newcommand{\Ll}{\langle}
 \newcommand{\Rr}{\rangle}

 \newcommand{\N}{\mathbb{N}}
 \newcommand{\R}{\mathbb{R}}

 \newcommand{\PP}{\mathbb{P}}

 \newcommand{\mcl}{\mathcal}
 
 \newcommand{\Be}{\begin{equation}}
 \newcommand{\Ee}{\end{equation}}
 \newcommand{\Bs}{\begin{split}}
 \newcommand{\Es}{\end{split}}
  \newcommand{\Bes}{\begin{equation*}}
 \newcommand{\Ees}{\end{equation*}}
 \newcommand{\BT}{\begin{thm}}
 \newcommand{\ET}{\end{thm}}
 \newcommand{\Bp}{\begin{proof}}
 \newcommand{\Ep}{\end{proof}}
 \newcommand{\BL}{\begin{lem}}
 \newcommand{\EL}{\end{lem}}
 \newcommand{\BP}{\begin{proposition}}
 \newcommand{\EP}{\end{proposition}}
 \newcommand{\BC}{\begin{corollary}}
 \newcommand{\EC}{\end{corollary}}
 \newcommand{\BR}{\begin{rem}}
 \newcommand{\ER}{\end{rem}}
 \newcommand{\BD}{\begin{defn}}
 \newcommand{\ED}{\end{defn}}
 \newcommand{\BI}{\begin{itemize}}
 \newcommand{\EI}{\end{itemize}}
 
 \newcommand{\tl}{\tilde}

\def\be#1{\begin{equation*}#1\end{equation*}}
\def\ben#1{\begin{equation}#1\end{equation}}
\def\bes#1{\begin{equation*}\begin{split}#1\end{split}\end{equation*}}
\def\besn#1{\begin{equation}\begin{split}#1\end{split}\end{equation}}

\newcommand{\eq}{\eqref}

\begin{document}
\title[Multivariate approximations in Wasserstein distance]
{Multivariate approximations in Wasserstein distance by Stein's method and Bismut's formula}

\author[X. Fang]{Xiao Fang}
\address{Department of Statistics,
The Chinese University of Hong Kong,
Shatin, N.T.,
Hong Kong}
\email{xfang@sta.cuhk.edu.hk }
\author[Q.M. Shao]{Qi-Man Shao}
\address{Department of Statistics,
The Chinese University of Hong Kong,
Shatin, N.T.,
Hong Kong}
\email{qmshao@sta.cuhk.edu.hk }
\author[L. Xu]{Lihu Xu}
\address{1. Department of Mathematics,
Faculty of Science and Technology,
University of Macau,
Av. Padre Tom\'{a}s Pereira, Taipa,
Macau, China. 
2. Zhuhai UM Science \& Technology Research Institute, Zhuhai, China}
\email{lihuxu@umac.mo}


\begin{abstract} \label{abstract}
Stein's method has been widely used for probability approximations.
However, in the multi-dimensional setting, most of the results are for multivariate normal approximation or for test functions with bounded second- or higher-order derivatives.
For a class of multivariate limiting distributions, we use Bismut's formula in Malliavin calculus to control the derivatives of the Stein equation solutions by the first derivative of the test function.
Combined with Stein's exchangeable pair approach, we obtain a general theorem for multivariate approximations with near optimal error bounds on the Wasserstein distance.
We apply the theorem to the unadjusted Langevin algorithm.

\medskip

\noindent{\bf AMS 2010 subject classification:} 60F05, 60H07.

\noindent{\bf Keywords and phrases:} Bismut's formula, Langevin algorithm, Malliavin calculus, multivariate approximation, rate of convergence, Stein's method, Wasserstein distance.

\end{abstract}

\maketitle

\section{Introduction}

Let $W$ and $Z$ be $d$-dimensional random vectors, $d\geq 1$,
where $Z$ has the density
\be{
Ke^{-U(x)}, \quad  x\in \mathbb{R}^d
}
for a given function $U: \mathbb{R}^d\to \mathbb{R}$ and a possibly unknown normalizing constant $K$.
We are concerned with bounding their Wasserstein distance, defined as follows:
\be{
d_{\mcl W}(\mathcal{L}(W), \mathcal{L}(Z)):\ = \ \sup_{h\in \text{Lip}(\mathbb{R}^d,1)} |\E[h(W)]-\E[h(Z)]|,
}
where $\text{Lip}(\mathbb{R}^d, 1)$ denotes the set of Lipschitz functions $h: \mathbb{R}^d\to \mathbb{R}$ with Lipschitz constant 1,
that is, $|h(x)-h(y)|\leq |x-y|$ for any $x, y\in \mathbb{R}^d$, and $|\cdot|$ denotes the Euclidean metric.

Our main tool is Stein's method for probability approximations \cite{St72}. Since it was first introduced,
there have been many developments in multivariate probability approximations.
However, most of the results are for multivariate normal approximation or for test functions $h$ with bounded second- or higher-order derivatives. See, for example, \cite{GoRi96, ChMe08, ReRo09, MaGo16}.
Although these results may be used to deduce error bounds for the Wasserstein distance,
such error bounds are far from optimal.
The literature on (near) optimal error bounds for the Wasserstein distance are limited to a few special cases, including
multivariate normal approximation for sums of independent random vectors \cite{VaVa10} and
multivariate approximation for the stationary distribution of certain Markov chains with bounded jump sizes \cite{Gu14, BrDa17}.

The main difficulty in obtaining an optimal error bound for the Wasserstein distance is controlling the derivatives of the Stein equation solutions using the first derivative of the test function $h$.
For multivariate non-normal approximations, the Stein equation solution is typically expressed in terms of a stochastic process (cf. \eq{e:SolRep}). This unexplicity means that
we cannot use the usual integration by parts formula when studying its derivatives.
This is in contrast to multivariate normal approximation, where we have an explicit expression of the Stein equation solution (cf. \eq{2}).
We use Bismut's formula (cf. \eq{de:IBP}) in Malliavin calculus to overcome this difficulty and obtain estimates for the derivatives of the Stein equation solutions for a large class of limiting distributions.
We note that Nourdin and Peccatti \cite{NoPe09, NoPe12} first combined Malliavin calculus and Stein's method to study normal approximation in a fixed Wiener chaos of a general Gaussian process. See \cite{NoPeRe10} and \cite{KuTu18} for generalizations to multivariate normal approximation and one-dimensional diffusion approximations, respectively.

The exchangeable pair is a powerful tool in Stein's method to exploit the dependence structure within the random vector $W$. It was elaborated in \cite{St86} and works for both independent and many dependent random vectors.
In particular, we use a generalized version in \cite{Ro08} and assume that we can construct a suitable random vector $W'$ on the same probability space and with the same distribution as $W$.
We then follow the idea of \cite{ChSh11} and \cite{ShZh16},
by studying the conditional expectations $\E[W'-W|W]$ and $\E[(W'-W)(W'-W)^{\rm T}|W]$ where ${\rm T}$ is the transpose operator, to identify the limiting distribution of $W$ and obtain an error bound for the Wasserstein distance in the approximation.
Our main result can be regarded as an extension of the result in \cite{ChSh11} to the multi-dimensional setting.
An additional logarithmic factor appears in our error bound due to the multi-dimensionality. We illustrate some of the techniques for removing it in the special case of multivariate normal approximation for standardized sums of independent and bounded random vectors.

Our main theorem can be used in justifying the so-called unadjusted Langevin algorithm \cite{RoTw96},
which is widely used in Bayesian inference and statistical physics to sample from a distribution that is known up to a normalizing constant.
In particular, we provide an error bound for the Wasserstein distance between the sampling distribution and the target distribution in terms of the step size in the algorithm.
Our result complements those in the literature by
relaxing the conditions on the increment distribution.

%


We would like to mention two other distances between distributions that have also been widely studied.
One is for comparing probabilities on convex sets in $\mathbb{R}^d$. See, for example, \cite{Go91, Be05} for multivariate normal approximation for sums of independent random vectors, and \cite{RiRo96} for sums of bounded random vectors with a certain dependency structure.
Proving optimal error bounds in this case requires special techniques involving smoothing of the test functions and induction or recursion.
The other distance is the so-called Wasserstein-2 distance which is stronger than the Wasserstein distance considered in this paper.
See, for example, \cite{LNP15, Bo18} and the references therein for related results.
The techniques used therein, which involves transportation inequalities and/or Stein kernels, are very different from ours.

The paper is organized as follows. In Section 2, we present our main result. In Section 3, we state the new properties of the Stein equation solutions and use them to prove our main result. The application to the unadjusted Langevin algorithm is discussed in Section 4. We develop Bismut's approach to Malliavin calculus to prove the properties of the Stein equation solutions in Sections 5--7. Some of the details are deferred to Section 8 and Appendix A. In Appendix B, we illustrate some of the techniques for removing the logarithmic term in our main result for the special case of multivariate normal approximation for sums of independent and bounded random vectors.

\section{Notation, assumptions and the main result}

\subsection{Notation}
The inner product of $x, y\in \R^d$ is denoted by $\Ll x,y\Rr$.
The Euclidean metric is denoted by $|x|$.
Each time we speak about Lipschitz functions on $\R^d$, we use the Euclidean norm.
$\mcl C(\R^d,\R)$ denotes the collection of all continuous functions $f: \R^d \rightarrow \R$ and $\mcl C^k(\R^d,\R)$ with $k \ge 1$ denotes the collection of all $k$-th order continuously differentiable functions.
$\mcl C^\infty_0(\R^d,\R)$ denotes the set of smooth functions whose every order derivative decays to zero at infinity.  For $f \in \mcl C^2(\R^d,\R)$ and $u, u_1, u_2, x \in \R^d$, the directional derivative $\nabla_u f(x)$ and  $\nabla_{u_2} \nabla_{u_1} f(x)$ are  defined by
\ \ \
\Bes
\begin{split}
& \nabla_u f(x)\ = \ \lim_{\e \rightarrow 0} \frac{f(x+\e u)-f(x)}{\e}, \\
\end{split}
\Ees
\Bes
\begin{split}
&\nabla_{u_2} \nabla_{u_1} f(x)\ = \ \lim_{\e \rightarrow 0} \frac{\nabla_{u_1}f(x+\e u_2)-\nabla_{u_1}f(x)}{\e},
\end{split}
\Ees
respectively.
Let $\nabla f(x)\in \R^d$, $\nabla^2 f(x)\in \R^{d \times d}$ and $\Delta f(x)\in \R$ denote the gradient, the Hessian and the Laplacian of $f$, respectively.
It is known that
$\nabla_u f(x) = \Ll \nabla f(x), u\Rr$ and
$\nabla_{u_2} \nabla_{u_1} f(x)=\Ll \nabla^2 f(x), u_1 u^{\rm T}_2\Rr_{{\rm HS}}$, where ${\rm T}$ is the transpose operator and $\Ll A, B\Rr_{{\rm HS}}:=\sum_{i,j=1}^d A_{ij} B_{ij}$ for $A, B \in \R^{d \times d}$. Given a matrix $A \in \R^{d \times d}$, its Hilbert-Schmidt norm is
$||A||_{{\rm HS}} \ = \ \sqrt{\sum_{i,j=1}^{d} A^{2}_{ij}} \ = \ \sqrt{{\rm Tr} (A^{\rm T} A)}$ and its operator norm is
$||A||_{op} \ = \ \sup_{|u|=1} |A u|$. We have the following relations:
\Be  \label{e:RelEigHS}
||A||_{op} \ = \ \sup_{|u_{1}|, |u_{2}|=1} |\Ll A, u_{1} u^{\rm T}_{2}\Rr_{\text{HS}}|, \ \ \ \ \ \ \ \ \ ||A||_{op} \ \le \ ||A||_{{\rm HS}} \ \le \ \sqrt d ||A||_{op} .
\Ee


We can also define $\nabla_u f(x)$ and $\nabla_{u_1} \nabla_{u_2} f(x)$ for a second-order differentiable function $f=(f_1,\dots, f_d)^{\rm T}: \R^d \rightarrow \R^d$ in the same way as above. Define $\nabla f(x)=(\nabla f_1(x), \dots, \nabla f_d(x)) \in \R^{d \times d}$ and $\nabla^2 f(x)=\{\nabla^2 f_i(x)\}_{i=1}^d \in \R^{d \times d \times d}$.
In this case, we have $\nabla_u f(x)=[\nabla f(x)]^{\rm T} u$ and $\nabla_{u_2} \nabla_{u_1} f(x)=\{ \Ll \nabla^2 f_1(x), u_1 u^{\rm T}_2\Rr_{{\rm HS}},\dots, \Ll \nabla^2 f_d(x), u_1 u^{\rm T}_2\Rr_{{\rm HS}}\}^{\rm T}$.

Moreover, $\mcl C_b(\R^{d_1}, \R^{d_2})$ denotes the set of all bounded measurable functions from $\R^{d_1}$ to $\R^{d_2}$ with the supremum norm defined by
$$\|f\|\ = \ \sup_{x \in \R^{d_1}} |f(x)|.$$
Denote by $C_{p_1,...p_k}$ some positive number depending on $k$ parameters, $p_1,...,p_k$, whose exact values can vary from line to line.

\subsection{Assumptions}
We aim to approximate $W$, a $d$-dimensional random vector of interest, by a non-degenerate probability measure $\mu$ on $\R^d$, which is the 
ergodic measure (cf. Remark \ref{r1}) of the following stochastic differential equation (SDE):
\
\Be  \label{de:OU}
\dif X_t\ = \ g(X_t) \dif t+\sqrt 2 \dif B_t,   \ \ \ \ \ \ \ X_0\ = \ x,
\Ee
where $B_t$ is a standard $d$-dimensional Brownian motion and
 \begin{assumption}  \label{a:A}
 $g\in \mcl C^2(\R^d,\R^d)$ and
there exist $\theta_0 > 0$ and $\theta_1, \theta_2, \theta_3\ge 0$ such that 
\Be \label{a2}
\Ll u, \nabla_u g(x)\Rr\ \le \ -\theta_0 \left(1+\theta_1 |x|^{\theta_2}\right)|u|^2, \ \ \ \ \ \ \forall \ u, x \in \R^d;
\Ee
\Be \label{a3}
|\nabla_{u_1} \nabla_{u_2} g(x)| \ \le \  \theta_3  (1+ \theta_1 |x|)^{\theta_2-1} |u_1| |u_2|, \ \ \ \ \ \forall \ u_1, u_2, x \in \R^d.
\Ee
\end{assumption}
\
\begin{remark}
By integration,
\eq{a2} implies
\Be \label{a1a}
\langle x, g(x)-g(0)\rangle \ \le \  -\theta_0 (|x|^2+ \frac{\theta_1 |x|^{2+\theta_2}}{1+\theta_2}), \ \ \ \ \ \forall \  x \in \R^d,
\Ee
and \eq{a3} implies
\Be \label{a1b}
|g(x)| \ \le \ \theta_4(1+|x|^{1+\theta_2}), \ \ \ \ \forall \ x \in \R^d,
\Ee
where $\theta_4>0$ is a constant depending on $\theta_1, \theta_2, \theta_3$, $g(0)$ and $\nabla g(0)$.
\end{remark}




\begin{remark}\label{r1}
The SDE \eq{de:OU} is known as the Langevin SDE.
The conditions $g\in \mcl C^1(\R^d,\R^d)$ and \eqref{a2} ensure that SDE \eqref{de:OU} has a unique strong solution, 
which hereafter is denoted by $X^x_t$.
This follows by Theorem 2.2 in \cite{DaGo01}.
In fact, (i) and (ii) of Hypothesis 2.1 in \cite{DaGo01} automatically hold, and (iii) is verified by
$$
\Ll g(x)-g(y), x-y\Rr \ = \ \int_0^1 \Ll \nabla_{x-y} g(\theta x+(1-\theta)y), x-y\Rr \dif \theta \ \le \ 0.
$$
Moreover, these two conditions also imply that the SDE admits a unique ergodic measure $\mu$ such that
\be{
\lim_{t \rightarrow \infty} \E f(X_t^x) \ = \  \mu(f)  
}
for all bounded continuous functions $f: \R^d \to \R$.
We leave the details to Appendix A.
\end{remark}

\begin{remark}
In the literature on multivariate probability approximations by $\mu$, it is often assumed that $\mu$ is strongly log-concave. See, for example, Eq. (1) in \cite{Da17}, Assumption H2 in \cite{DuMo18} and Theorem 2.1 in \cite{MaGo16}.
This corresponds to condition \eq{a2} with $\theta_0>0$ and $\theta_1=0$.
We use this condition to ensure the existence of the Stein equation solution (cf. \eq{e:SolRep}).
Moreover, because we will use integration by parts (under Malliavin calculus), we need to impose conditions on the higher-order derivatives of $g$, such as \eq{a3}.
\end{remark}

Below, we give two examples of $\mu$ that satisfy Assumption \ref{a:A}
and one counterexample.

\noindent {\bf Example 1}: $g(x)=-Ax$ and $A \in \R^{d \times d}$ is a symmetric positive definite matrix with eigenvalues $\lambda_1\geq \dots \geq \lambda_d>0$. The corresponding measure $\mu$ is a Gaussian measure with the density $$\varphi(x)=\left(\frac{\det(A)}{2 \pi}\right)^{d/2} \exp\left(-\frac{\Ll x, A x\Rr}{2}\right).$$
It can be verified that Assumption \ref{a:A} is satisfied with $\theta_0=\lambda_d, \theta_1=0, \theta_2=1, \theta_3=1$, and \eq{a1b} holds with $\theta_4=\lambda_1.$ \\


\noindent {\bf Example 2}:
$g(x)=-c(1+|x|^2)^{p/2} x$ with $p \ge 0$ and $c>0$. The corresponding measure $\mu$ has the density function
$$K e^{-\frac{c}{p+2} (1+|x|^2)^{p/2+1}}  \ \ \ {\rm with} \ \ \  K=\left(\int_{\R^d} e^{-\frac{c}{p+2} (1+|x|^2)^{p/2+1}} \dif x\right)^{-1}.$$
In this case, we have
$$\Ll u, \nabla_{u} g(x)\Rr\ = \ -c(1+|x|^2)^{\frac p2} |u|^2-c p(1+|x|^2)^{\frac{p-2}2} |\Ll x,u\Rr|^2 ,$$
\Bes
\begin{split}
\nabla_{u_2} \nabla_{u_1} g(x) \ = & \ -cp(1+|x|^2)^{\frac p2-1} \left(\Ll x,u_1\Rr u_2+\Ll x,u_2\Rr u_1+\Ll u_1,u_2\Rr x\right) \\
& \ -cp(p-2)(1+|x|^2)^{\frac p2-2} \Ll x,u_1\Rr\Ll x,u_2\Rr x.
\end{split}
\Ees
It is then straightforward to determine $\theta_0,\dots, \theta_4$ which depend on $c$ and $p$.
We omit the details.

\medskip

\noindent{\bf Counterexample}: $g(x)=-c|x|^p x$ with $p\ge 1$ and $c>0$. We have
$$\Ll u, \nabla_{u} g(x)\Rr\ = \ -c\left(|x|^{p} |u|^{2}+p|x|^{p-2} |\Ll x, u\Rr|^{2}\right),$$
which does not satisfy \eq{a2} with any positive $\theta_0$.

\subsection{Main result}
We use a version of Stein's exchangeable pair approach by R\"ollin \cite{Ro08} to exploit the dependence structure of the random vector $W$ of interest with $\E |W|<\infty$.
Suppose we can construct a suitable random vector $W'$ on the same probability space and with the same distribution as $W$.
Denote $\delta=W'-W$ and assume that
\
\Be  \label{e:A1}
\E\left[\delta|W\right]\ = \ \lambda (g(W)+R_1),
\Ee
where $R_1$ is an $\R^d$-valued random vector. Further assume that
\
\Be  \label{e:A2}
\E \left[\delta \delta^{\rm T}|W\right]\ = \ 2\lambda (I_d+R_2)
\Ee
where $I_d$ denotes the $d\times d$ identity matrix. In some applications, $R_{2}$ has the form
\Be \label{e:SpeR2}
R_{2} \ = \ r_{1} r^{\rm T}_{2}+...+r_{2p-1} r^{\rm T}_{2p},
\Ee
where $p \in \N$ and $r_{1},...,r_{2p}\in \R^d$. In the application to the unadjusted Langevin algorithm in Section 4, $p=1$.

Our main result is the following theorem on multi-dimensional non-normal (including normal) approximations.
\ \ \ \
\begin{thm}  \label{t:MThm}
Let Assumption \ref{a:A}, \eqref{e:A1} and \eqref{e:A2} hold. Then we have
\Be  \label{e:EstDel-0}
\begin{split}
d_{\mcl W}(\mathcal{L}(W), \mu)& \ \le \ C_\theta  \bigg\{ \frac{1}{\lambda}\E\left[|\delta|^3 \left(|\log|\delta|| \vee 1\right)\right]+\E |R_1|+\sqrt d \E\left[||R_{2}||_{\rm HS}\right]\bigg\},
\end{split}
\Ee
where $\mu$ is the ergodic measure of SDE \eqref{de:OU} and hereafter $C_\theta$ is short hand for $C_{\theta_0,...,\theta_4}$. If $R_{2}$ has the form \eqref{e:SpeR2}, then
\Be  \label{e:EstDel}
\begin{split}
d_{\mcl W}(\mathcal{L}(W), \mu)& \ \le \ C_\theta  \bigg\{ \frac{1}{\lambda}\E\left[|\delta|^3 \left(|\log|\delta|| \vee 1\right)\right]+\E |R_1|+\sum_{i=1}^{p}\E\left[|r_{2i-1}||r_{2i}|\right]\bigg\}.
\end{split}
\Ee
\end{thm}
\ \ \

\begin{remark}
Note that when $\theta_0,\dots, \theta_3$ in Assumption \ref{a:A} and $\theta_4$ in \eq{a1b} are independent of the dimension $d$ as in the two examples above, the constant $C_\theta$ in \eq{e:EstDel-0} and \eq{e:EstDel} is dimension-free.
\end{remark}

\begin{remark}
Theorem \ref{t:MThm} can be regarded as an extension of \cite[Lecture 3, Theorem 1]{St86} and \cite[Theorem 1.1]{ChSh11} to the multi-dimensional and non-normal approximations,
with a minor cost of an additional logarithmic factor.
\end{remark}

\begin{remark}
If, instead of $\E \left[\delta \delta^{\rm T}|W\right]\approx 2\lambda I_d$ in \eq{e:A2}, we have $\E \left[\delta \delta^{\rm T}|W\right]\approx 2\lambda \Lambda$ for an invertible, positive definite matrix $\Lambda$, then we may approximate it by the ergodic measure of the following SDE with non-identity diffusion coefficient
\be{
\dif X_t\ = \ g(X_t) \dif t+\sqrt 2 \Lambda^{1/2} \dif B_t,   \ \ \ \ \ \ \ X_0\ = \ x.
}
Our approach is still applicable to this case, although details need to be work out with greater effort, especially if we
are interested in the dependence of the bound on $\Lambda$.
We may as well reduce the problem to the setting of Theorem \ref{t:MThm} by considering approximating $\Lambda^{-1/2} W$ by the ergodic measure of the SDE with identity diffusion coefficient and drift coefficient as
\be{
\widetilde{g}(\cdot)=\Lambda^{-1/2} g (\Lambda^{1/2}\cdot),
}
although then the problem becomes obtaining an explicit expression of $C_\theta$ in Theorem \ref{t:MThm} in terms of the parameters appearing in Assumption \ref{a:A}.
\end{remark}

\begin{remark}\label{r8}
In the case of multivariate normal approximation for sums of independent, bounded random variables $W=\frac{1}{\sqrt{n}}\sum_{i=1}^n X_i$ with $\E X_i=0, |X_i|\leq \beta$ and $\E W W^{\rm T}=I_d$, the bound in \eq{e:EstDel-0} reduces to, with details deferred to Appendix B,
\be{
 \frac{C d \beta}{\sqrt{n}}(1+\log n),
}
where $C$ is an absolute constant.
This is of the same order as in Theorem 2 of \cite{VaVa10}.
In this case, we may remove the additional logarithmic factor and obtain the error bound (see Appendix B)
\be{
\frac{Cd^2 \beta}{\sqrt{n}}.
}
The additional logarithmic term may also be removed for $W$ to be a sum of independent and unbounded random vectors and for $W$ to exhibit an exchangeable pair.
However, we would need certain moment assumptions and increase the dependence of the bound on the dimension.
Also in this case, under the additional assumption that $\E X_i X_i^{\rm T}=I_d$ for each $i=1,\dots, n$,
Bonis \cite{Bo18} obtained the optimal rate $O(\frac{\sqrt{d} \beta}{\sqrt{n}})$ in the stronger Wasserstein-2 distance, which seems better than the previous results of \cite{CoFaPa17} and \cite{Zh17}. Moreover, his general result \cite[Theorem 1]{Bo18} extended those in \cite{Sak85,Bob18}, and improved the multidimensional bound in \cite{Zh17} by removing some boundedness assumption and an additional $\log n$ factor therein. 
We do not know how to obtain their results using our approach.
\end{remark}


\section{Stein's method and the proof of Theorem \ref{t:MThm}} \label{s:Stein}
Let $g$ satisfy Assumption \ref{a:A}, and let $\mu$ be the ergodic measure of of SDE \eqref{de:OU}.
Then $\mu$ is invariant in the sense that
\ \
$$\int_{\R^d} \E f(X^x_t) \mu(\dif x)\ = \ \int_{\R^d} f(x) \mu(\dif x), \ \ \ \ \ \  f \in \mcl C^\infty_0(\R^d, \R).$$
It is well known that $\mu$ satisfies the following equation
\ \
\be{  
\int_{\R^d} \left[\Delta f(x)+\Ll g(x), \nabla f(x)\Rr\right] \mu(\dif x)\ = \ 0, \ \ \ \ \ \ f \in \mcl C^\infty_0(\R^d, \R).
}
See \cite[p. 326]{ABR99} for more details.

For a Lipschitz function $h:\R^d \rightarrow \R$, consider the Stein equation
\
\Be \label{de:StEq}
\Delta f+\Ll g(x), \nabla f(x)\Rr\ = \ h(x)-\mu(h),
\Ee
where $\mu(h):=\int_{\R^d} h(x) \mu(dx)$ which exists (cf. \eq{e:2ndMom}).
The solution $f:=f_h$ exists and we drop the subscript for ease of notation.
The following theorem on the regularity of $f$ is crucial for the proof of our main result.

\begin{theorem}  \label{t:D2FEst}
Let $h: \R^d \rightarrow \R$ be a Lipschitz function and let $\e \in \R$. For any $u_{1}, u_{2} \in \R^{d}$, we have
\ \ \
\Be \label{e:f'Est>}
\|\nabla f\| \ \le \  C_{\theta} \|\nabla h\|,
\Ee
\Be \label{e:D2FEst>}
\sup_{x \in \R^{d}}|\Ll \nabla^2 f(x), u_{1}u^{\rm T}_{2}\Rr_{{\rm HS}}| \ \le \ C_{\theta}  \|\nabla h\| |u_{1}||u_{2}|,
\Ee
\Be \label{e:D2FEst2>}
\sup_{x, u \in \R^{d}, |u| \le 1}|\Ll \nabla^2 f(x+\e u)-\nabla^2 f(x), u_{1}u^{\rm T}_{2}\Rr_{{\rm HS}}|  \ \le \  C_{\theta}  ||\nabla h|| |\e| \left(|\log |\e|| \vee 1\right) |u_{1}||u_{2}|.
\Ee
\end{theorem}
\begin{rem}
Gorham et. al. \cite{Go17} recently put forward a method to measure sample quality with diffusions by a Stein discrepancy, in which the same Stein equation as \eqref{de:StEq} has to be considered. Under the assumption that $g$ is {\em 3rd order differentiable},  they used the Bismut-Elworthy-Li formula \cite{ElLi94}, together with smooth convolution and interpolation techniques, to prove a bound on the first, second and $(3-\epsilon)$th derivative of $f$ for $\epsilon>0$.
They can also obtain the bound \eqref{e:D2FEst2>} by their approach (personal communication \cite{Les18}), albeit due to the interpolation argument therein, the assumption of 3rd order differentiability of $g$ can not be removed. 
\end{rem}

We defer the proof of Theorem \ref{t:D2FEst} to Sections 6 and 7 by deriving stochastic representations of $f$ and its derivatives as follows:
\Bes
f(x)\ = \ \int_0^\infty e^{-t} \E[f(X_t^x)+\mu(h)-h(X_t^x)] \dif t,
\Ees
\Bes
\begin{split}
\nabla_u f(x)\ = \ \int_0^\infty e^{-t} \E\left\{\big[f(X_t^x)-h(X_t^x)+\mu(h)\big]   \mcl I_u^x(t)\right\} \dif t,
\end{split}
\Ees
where $X_t^x$ is the stochastic process determined by SDE \eqref{de:OU} and $\mcl I_u^x (t)$ is a stochastic integral. The representation of $\nabla^2 f(x)$ is more complicated and can be found in \eq{e:g''4}.
\\

With the regularity result in Theorem \ref{t:D2FEst}, we are in a position to prove our main result.
\begin{proof} [{\bf Proof of Theorem \ref{t:MThm}}]
From the fact that $W$ and $W'$ have the same distribution and using Taylor's expansion, we have
\ \
\Bes
\begin{split}
0&\ = \ \E\left[f(W')-f(W)\right] \\
&\ = \ \E\left[\Ll \delta, \nabla f(W) \Rr\right]+\frac 12  \int_0^1\E\left[\langle  \delta \delta^{\rm T}, \nabla^2  f(W+t \delta) \rangle_{{\rm HS}}\right] \dif t.
\end{split}
\Ees
By \eq{e:A1} and \eq{e:A2}, we have
\ \ \
\Bes
\begin{split}
\E\left[\Ll \delta, \nabla f(W) \Rr\right]&\ = \ \E \left[\Ll\E\left(\delta|W\right), \nabla f(W)\Rr\right]  \\
&\ = \ \lambda \E \left[\Ll g(W), \nabla f(W)\Rr\right]+ \lambda \E \left[\Ll R_1, \nabla f(W)\Rr\right]
\end{split}
\Ees
and
\ \ \
\Bes
\begin{split}
& \ \ \ \ \int_0^1\E\left[\langle  \delta \delta^{\rm T}, \nabla^2  f(W+t \delta) \rangle_{{\rm HS}}\right] \dif t \\
&\ = \ \E\left[\langle  \delta \delta^{\rm T}, \nabla^2  f(W) \rangle_{{\rm HS}}\right]+\int_0^1 \E\left[\langle  \delta \delta^{\rm T}, \nabla^2  f(W+t \delta)-\nabla^2  f(W)  \rangle_{{\rm HS}}\right] \dif t \\
&\ = \ \E\left[\langle  \E\left(\delta \delta^{\rm T}|W\right), \nabla^2  f(W) \rangle_{{\rm HS}}\right]+\int_0^1 \E\left[\langle  \delta \delta^{\rm T}, \nabla^2  f(W+t \delta)-\nabla^2  f(W)  \rangle_{{\rm HS}}\right] \dif t \\
&\ = \ 2 \lambda \E\left[\Delta  f(W)\right]+ 2\lambda \E\left[\langle R_2, \nabla^2f(W)\rangle_{{\rm HS}}\right]+\int_0^1 \E\left[\langle  \delta \delta^{\rm T}, \nabla^2  f(W+t \delta)-\nabla^2  f(W)  \rangle_{{\rm HS}}\right] \dif t.
\end{split}
\Ees
Combining the previous three equations, we obtain
\ \ \
\Bes
\begin{split}
 \E\left[\Delta f(W)+\Ll g(W), \nabla f(W)\Rr\right]
&\ = \ -\E \left[\Ll R_1, \nabla f(W)\Rr\right]-\E\left[\langle R_2, \nabla^2f(W)\rangle_{{\rm HS}}\right] \\
&\ \ \ \ \ -\frac{1}{2\lambda} \int_0^1 \E\left[\langle  \delta \delta^{\rm T}, \nabla^2  f(W+t \delta)-\nabla^2  f(W)  \rangle_{{\rm HS}}\right] \dif t.
\end{split}
\Ees
By \eqref{de:StEq}, we have
\ \ \ \
\Bes
\begin{split}
 \left|\E[h(W)]-\mu(h)\right|& \ \le \ \left|\E\left\{\Ll R_1, \nabla f(W)\Rr\right\}\right|+\left|\E\left[\langle R_2, \nabla^2 f(W)\rangle_{{\rm HS}}\right]\right| \\
&+\frac{1}{2\lambda}\int_0^1 \left|\E\left[\langle  \delta \delta^{\rm T}, \nabla^2  f(W+t \delta)-\nabla^2  f(W)  \rangle_{{\rm HS}}\right]\right| \dif t.
\end{split}
\Ees
By \eqref{e:f'Est>}, we have
\ \ \
\Bes
 \left|\E\left[\Ll R_1, \nabla f(W)\Rr\right]\right| \ \le \ C_\theta \|\nabla h\| \E\left[|R_1|\right].
\Ees
By \eq{e:D2FEst>} and \eqref{e:RelEigHS}, we have
 \Be
  \left|\E\left[\langle R_2, \nabla^2 f(W)\rangle_{{\rm HS}}\right]\right| \ \le \   \sup_{x \in \R^{d}}||\nabla^2 f(x)||_{{\rm HS}}\E\left[||R_2||_{{\rm HS}}\right] \ \le \ C_{\theta} \sqrt d \|\nabla h\| \E\left[||R_2||_{{\rm HS}}\right].
 \Ee
If $R_{2}$ has the form \eqref{e:SpeR2}, by \eqref{e:D2FEst>} we have
 \Be
  \left|\E\left[\langle R_2, \nabla^2 f(W)\rangle_{{\rm HS}}\right]\right| \ = \ \left|\E\left[\sum_{i=1}^{p} \Ll r_{2i-1} r^{\rm T}_{2i}, \nabla^2 f(W)\rangle_{{\rm HS}}\right]\right|\ \le \ C_\theta \|\nabla h\| \sum_{i=1}^{p}\E\left[|r_{2i-1}||r_{2i}|\right].
\Ee
Moreover, by \eqref{e:D2FEst2>} we have
\ \ \ \
\Be\label{101}
\begin{split}
& \ \ \ \int_0^1 \left|\E\left[\langle  \delta \delta^{\rm T}, \nabla^2  f(W+t \delta)-\nabla^2  f(W)  \rangle_{{\rm HS}}\right]\right| \dif t \\
&\ = \ \int_0^1 \left|\E\left[\left\langle  \delta \delta^{\rm T}, \nabla^2  f\left(W+|\delta| \frac{t \delta}{|\delta|}\right)-\nabla^2  f(W)  \right\rangle_{{\rm HS}}\right]\right| \dif t \\
& \ \le \ C_\theta  \|\nabla h\| \E\left[|\delta|^3 \left(|\log|\delta|| \vee 1\right)\right].
\end{split}
\Ee
Combining the inequalities above, we obtain \eqref{e:EstDel-0} and \eqref{e:EstDel}.
\end{proof}

\section{An application: Unadjusted Langevin algorithm}


We consider the problem of sampling a probability distribution $\mu$ that has the density
\be{
Ke^{-U(x)},
}
where $U(x)$ is a given function, but the normalizing constant $K$ is unknown.
This problem is encountered in Bayesian inference, where $\mu$ is a posterior distribution, and in statistical physics, where $\mu$ is the distribution of particle configurations.
As $K$ is unknown, we cannot sample from $\mu$ directly.
The so-called unadjusted Langevin algorithm (ULA) with fixed step size is as follows.
We refer to \cite{RoTw96, Da17} and the references therein for more details.
Regard $\mu$ as the stationary distribution of the Langevin stochastic differential equation
\be{
dX_t\ = \ g(X_t) dt+\sqrt{2} dB_t,
}
where $g(\cdot)=-\nabla U(\cdot)$ and $B_t$ is a standard $d$-dimensional Brownian motion.
The Euler-Maruyama discretization of $X_t$ with step size $s$ is
\ben{\label{11}
Y_{k+1}\ = \ Y_k+s g(Y_k)+\sqrt{2s} Z_{k+1},
}
where $Y_0$ is an arbitrary initial value and $Z_1, Z_2, \dots$ are independent and identically distributed standard $d$-dimensional Gaussian random vectors. See Remark \ref{r2} below for other possible choices of $\{Z_i\}$. We assume that $\{Y_k\}$ has an invariant measure $\mu_s$. The existence of $\mu_s$ has been extensively studied in the literature.
See, for example, \cite{DWT95, RoTw96, Da17}.
In particular, Dalalyan \cite{Da17} showed that $\mu_s$ exists for sufficiently small $s$, provided that $\mu$ is strongly log-concave, and $g$ is Lipschitz.
The so-called ULA with fixed step size uses the Markov chain Monte Carlo method to sample from $\mu_s$, then claims that $\mu_s$ is close enough to $\mu$ for a small $s$.

There is a tradeoff in the choice of step size $s$.
When $s$ becomes smaller, $\mu_s$ is closer to $\mu$, but it takes longer for the Markov chain to reach stationarity, and vice versa.
Therefore, it is of interest to quantify the distance between $\mu_s$ and $\mu$ for a given $s$.

Using our general theorem, we obtain the following result. The step size $s$ is typically small, and for ease of presentation, we assume that $s<1/e$.

\begin{theorem}
Under the above setting,
suppose $g(\cdot)$ satisfies Assumption \ref{a:A}.
For $s<1/e$, we have
\besn{\label{12}
d_{\mcl W}(\mu_s, \mu)\ \le \  & C_\theta
\sqrt s\left\{|\log s|\E |Z_1|^3+s^{3/2} |\log s| \E |g(W)|^3+s^{1/2}  \E |g(W)|^2+{\E\left[|\tl \delta|^3 |\log |\tilde \delta||\right]}\right\},
}
where $W\sim \mu_s$, $W$ is independent of $Z_1$, and $\tl \delta =\sqrt s g(W)+\sqrt{2} Z_1$.
\end{theorem}

\begin{remark}\label{r2}
As $s\to 0$, the leading-order term in the upper bound of \eq{12} decays as $d^{3/2} s^{1/2} $ up to a logarithmic factor.
Upper bounds between $\mu_s$ and $\mu$ for the stronger Wasserstein-2 distance have been obtained in the literature.
 See, for example, \cite{Da17, DuMo18, Bo18}.
 In particular, \cite[Corollary 9]{DuMo18} obtained a bound of the order $O(d s)$ under a slightly different set of conditions on $g$, which does not cover, say, Example 2 below Assumption \ref{a:A}.
 Their bound shows a lower computational complexity to achieve certain precision of the ULA.
A possible way to improve our bound, which holds as long as $\{Z_i\}$ has mean 0 and covariance matrix $I_d$, is to do another Taylor's expansion in \eq{101} and make use of the symmetry condition of the Gaussian (or Rademacher as in \cite{Bo18}) vector $Z_1$, that is, $\E Z_{1i}Z_{1j}Z_{1k}=0$ for any $i,j,k\in \{1,\dots, d\}$.

\end{remark}




\begin{proof}
Suppose $Y_0\sim \mu_s$ in \eq{11}.
Let $W=Y_0$ and $W'=Y_{1}$.
Because $\mu_s$ is the stationary distribution of the Markov chain \eq{11},
$W$ and $W'$ have the same distribution.
With
\be{
\delta\ = \ W'-W\ = \ sg(W)+\sqrt{2s} Z_{1},
}
we have
\be{
\E(\delta|W)\ = \ sg(W)
}
and
\be{
\E(\delta \delta^{\rm T}|W)\ = \ 2s I_d+ s^2 g(W)g^{\rm T}(W).
}
In applying Theorem \ref{t:MThm}, 
\be{
\lambda\ = \ s, \quad R_1\ = \ 0,\quad R_{2}=\frac s 2  g(W)g^{\rm T}(W),
}
thus,  $p=1$, $r_{1}=g(W)$ and $r_{2}=\frac s2 g^{\rm T}(W)$. 
We have
\be{
 \E \left[|r_{1}||r_{2}|\right]\ \le \  \frac{s}{2} \E|g(W)|^2.
}
Recall $\tl \delta=\frac{\delta}{\sqrt s}$ and write $g:=g(W)$. We have
\be{
\begin{split}
\frac{1}{\lambda}\E [|\delta|^3(|\log |\delta||\vee 1)]& \ \le \ \sqrt s \E\left[|\tl \delta|^3(|\log s|+|\log |\tilde \delta||)\right] \\
&\ \le \ \sqrt s|\log s| \E |\tl \delta|^3+\sqrt s \E\left[|\tl \delta|^3 |\log |\tilde \delta||\right].
\end{split}
}
Moreover,
$$\E |\tl \delta|^3 \ = \ \E |\sqrt s g+Z_1|^3 \ \le \ 4 s^{3/2} \E |g|^3+4 \E |Z_1|^3;$$
hence,
\be{
\begin{split}
\frac{1}{\lambda}\E [|\delta|^3(|\log |\delta||\vee 1)]& \ \le \ \sqrt s \E\left[|\tl \delta|^3(|\log s|+|\log |\tilde \delta||)\right] \\
&\ \le \ 4 \sqrt s\left\{s^{3/2}|\log s|  \E |g|^3+|\log s|\E |Z_1|^3+\E\left[|\tl \delta|^3 |\log |\tilde \delta||\right]\right\}.
\end{split}
}
The theorem is proved by applying the above bounds in \eqref{e:EstDel}.
\end{proof}

\section{Preliminary: Malliavin calculus of SDE \eqref{de:OU}} \label{s:SDE}
From this section, we start our journey toward proving the crucial Theorem \ref{t:D2FEst}. We use Bismut's approach to Malliavin calculus.
To this end, we first provide a brief review of Malliavin calculus in this section; the proofs of the related lemmas are deferred to Section \ref{s:ProofSect5}.
Throughout the remaining sections, let $X_t^x$ be the solution to SDE \eqref{de:OU}, where $g$ satisfies Assumption \ref{a:A}.
\subsection{Jacobi flow associated with SDE \eqref{de:OU} (\cite{Sa01})}
We consider the derivative of $X_t^x$ with respect to initial value $x$, which is called the Jacobian flow.
Let $u \in \R^d$,  the Jacobian flow $\nabla_u X_t^x$ along the direction $u$ is defined by
\
\be{
\nabla_u X_t^x\ = \ \lim_{\e \rightarrow 0} \frac{X_t^{x+\e u}-X_t^x}{\e}, \ \ \ \ t \ge 0.
}
The above limit exists and satisfies
\
\Be  \label{e:DuXt}
\frac{\dif}{\dif t} \nabla_u X_t^x\ = \ \nabla g(X_t^x) \nabla_u X_t^x, \ \ \ \ \ \nabla_u X_0^x\ = \ u,
\Ee
which is solved by
\
\Be  \label{de:Jxt}
\nabla_u X_t^x\ = \ \exp\left\{\int_0^t \nabla g(X_r^x) \dif r\right\} u.
\Ee
For further use, we denote
\ \ \
\Be  \label{e:JstX}
J_{s,t}^{x}\ = \ \exp\left\{\int_s^t \nabla g(X_r^x) \dif r\right\}, \ \ \ 0 \le s \le t<\infty.
\Ee
It is easy to see that $J_{s,t}^{x} J_{0,s}^{x}=J_{0,t}^{x}$ for all $0 \le s \le t<\infty$ and
\Bes
\begin{split}
\nabla_u X_t^x\ = \ J_{0,t}^xu.
\end{split}
\Ees
For $u_1,u_2 \in \R^d$, we can similarly define $\nabla_{u_2} \nabla_{u_1}  X_t^x$, which satisfies
\
\Be  \label{e:Du12Xt}
\frac{\dif}{\dif t} \nabla_{u_2}\nabla_{u_1}  X_t^x\ = \ \nabla g(X_t^x)\nabla_{u_2}\nabla_{u_1}  X_t^x+\nabla^2 g(X_t^x) \nabla_{u_2} X_t^x\nabla_{u_1} X_t^x
\Ee
with $\nabla_{u_2}\nabla_{u_1}  X_0^x=0$.  \\

The following lemmas give estimates of $X_t^x$, $\nabla_{u_1}  X_t^x$ and $\nabla_{u_2}\nabla_{u_1}  X_t^x$ and the proofs are given in Section 8.
\
\begin{lem} \label{l:SolPro}
We have
\
\be{ 
\E|X_t^x|^2 \ \le \ e^{- \theta_{0} t} |x|^{2}+\frac{2d+|g(0)|^2/\theta_0}{\theta_{0}}.
}
This further implies that the ergodic measure $\mu$ has finite 2nd moment and
\Be \label{e:2ndMom}
\begin{split}
\int_{\R^{d}} |x|^2 \mu(\dif x) \   \le \  \frac{2d+|g(0)|^2/\theta_0}{\theta_{0}}.
\end{split}
\Ee
\end{lem}
\begin{lem}  \label{l:JacDev}
For all $u_1,u_2 \in \R^d$ and $x \in \R^d$, we have the following (deterministic) estimates:
\Be \label{de:Jac1Est>}
|\nabla_{u_1}  X_t^x|  \ \le \  e^{-\theta_{0} t}|u_1|,
\Ee
\Be \label{de:Jac2Est>}
|\nabla_{u_2}\nabla_{u_1}  X_t^x|  \ \le \ C_{\theta}  |u_1| |u_2|.
\Ee
\end{lem}
\vskip 3mm
\subsection{Bismut's approach to Malliavin calculus for SDE \eqref{de:OU} (\cite{No86})}
Let $v \in L_{loc}^2([0,\infty) \times (\Omega, \mcl F, \PP);\R^d)$, i.e., $\E\int_0^t |v(s)|^2 ds<\infty$ for all $t>0$.
Further assume that $v$ is adapted to
the filtration $(\mcl F_{t})_{t \ge 0}$ with $\mcl F_t:=\sigma(B_s:0 \le s \le t)$; i.e., $v(t)$ is $\mcl F_t$ measurable for $t \ge 0$.
Define
\Be  \label{e:VtDef}
V_{t}\ =\ \int_0^t v(s) \dif s, \ \ \ \ \ \ t \ge 0.
\Ee
For a $t>0$, let $F_{t}: \mcl C([0,t],\R^{d}) \rightarrow \R$ be a $\mcl F_{t}$ measurable map. If the following limit exists
\Bes
D_V F_{t}(B)\ =\ \lim_{\e \rightarrow 0}\frac{F_{t}(B+\e V)-F_{t}(B)}{\e}
\Ees
in $L^2((\Omega,\mcl F,\PP);\R)$, then $F_{t}(B)$ is said to be \emph{Malliavin differentiable} and $D_VF_{t}(B)$ is called the Malliavin derivative of
$F_{t}(B)$ in the direction $v$; see \cite[p. 1011]{HaMa06}.

Let $F_{t}(B)$ and $G_{t}(B)$ both be Malliavin differentiable, then the following product rule holds:
\Be  \label{de:ProRul}
D_V(F_{t}(B)G_{t}(B))\ =\ F_t(B) D_V G_t(B)+G_t(B) D_V F_t(B).
\Ee
When
$$F_{t}(B)\ = \ \int_0^t \Ll a(s), \dif B_s\Rr,$$
where $a(s)=(a_1(s),...,a_d(s))$ is a deterministic function such that $\int_{0}^{t} |a(s)|^{2} \dif s<\infty$ for all $t>0$, it is easy to check that
\Bes
D_V F_{t}(B)\ =\ \int_0^t \Ll a(s),  v(s)\Rr \dif s.
\Ees
If $a(s)=(a_1(s),...,a_d(s))$ is a $d$-dimensional stochastic process adapted to the filtration $\mcl F_s$ such that $\E \int_{0}^{t} |a(s)|^{2} \dif s<\infty$ for all $t>0$, then
\Be  \label{de:DVInt}
D_V F_{t}(B)\ =\ \int_0^t \Ll a(s),  v(s)\Rr \dif s+\int_0^t \Ll D_V a(s), \dif B_s\Rr.
\Ee
\ \ \

The following integration by parts formula, called Bismut's formula, is probably the most important property in Bismut's approach to Malliavin calculus.

\

\noindent{\bf Bismut's formula.} For Malliavin differentiable $F_t(B)$ such that $F_t(B), D_V F_t(B)\in L^2((\Omega,\mcl F,\PP);\R)$, we have
\Be  \label{de:IBP}
\E\left[D_V F_t(B)\right]\ =\ \E\left[F_t(B)\int_0^t \Ll v(s),  \dif B_s\Rr\right].
\Ee

\

Let $\phi: \R^d \rightarrow \R$ be Lipschitz and let $F_t(B)=(F^{1}_t(B),...,F^{d}_t(B))$ be a $d$-dimensional Malliavian differentiable functional. The following chain rule holds:
\
\be{
D_V \phi(F_{t}(B))\ =\ \Ll \nabla  \phi(F_{t}(B)), D_V F_{t}(B)\Rr\ = \ \sum_{i=1}^d \p_i \phi(F_{t}(B)) D_V F^{i}_t(B).
}

Now we come back to SDE \eqref{de:OU}. Fixing $t \ge 0$ and $x \in \R^d$, the solution
$X_t^x$ is a $d$-dimensional functional of Brownian motion $(B_s)_{0 \le s \le t}$.
The following Malliavin derivative of
$X^{x}_t$ along the direction $V$ exists in $L^2((\Omega,\mcl F,\PP);\R^{d})$ and is defined by
\Bes
D_V X^{x}_t(B)\ =\ \lim_{\e \rightarrow 0}\frac{X^{x}_t(B+\e V)-X^{x}_t(B)}{\e}.
\Ees
We drop the $B$ in $D_V X^x_t(B)$ and write $D_V X^x_t=D_V X^x_t(B)$ for simplicity. It satisfies the equation
\
\Bes
\dif D_V X_t^x\ =\ \nabla g(X_t^x) D_V X_t^x \dif t+\sqrt 2 v(t) \dif t, \ \ \ D_V X^{x}_0=0,
\Ees
and the equation has a unique solution:
\Be \label{de:DVXt}
D_V X_t^x\ =\ \sqrt 2 \int_0^t J_{r,t}^x v(r) \dif r,
\Ee
where $J_{r,t}^{x}$ is defined by \eqref{e:JstX}.
Comparing \eqref{de:Jxt} and \eqref{de:DVXt}, if we take
\
\Be \label{de:vs}
v(s)\ =\ \frac 1{\sqrt 2 t} \nabla_u X_s^x, \ \ \ \ \ 0 \le s \le t,
\Ee
$\big($recall \eq{de:Jac1Est>} and $V_{t}=\int_{0}^{t} v(s) \dif s$$\big)$, because $\nabla_u X_r^x=J_{0,r}^x u$ and $J^{x}_{r,t} J^{x}_{0,r}=J^{x}_{0,t}$ for all $0 \le r \le t$,
we have
\ \ \ \
\Be  \label{de:DV=J}
D_V X_t^x\ =\ \nabla_u X_t^x
\Ee
and
\ \ \ \
\Be  \label{de:DV=Js}
D_V X_s^x\ =\ \frac s t\nabla_u X_s^x, \ \ \ 0 \le s \le t.
\Ee

\noindent Let $u_1, u_2 \in\R^d$, and define $v_i$ and $V_i$ as \eqref{de:vs} and \eqref{e:VtDef}, respectively, for $i=1,2$.
We can similarly define $D_{V_2} \nabla_{u_1} X_s^x$, which satisfies the following equation: for $s \in [0,t]$,
\Be \label{003}
\begin{split}
\frac{\dif}{\dif s}D_{V_2} \nabla_{u_1} X_s^x&\ = \ \nabla g(X_s^x) D_{V_2} \nabla_{u_1} X_s^x+\nabla^2 g(X_s^x) D_{V_2} X_s^x \nabla_{u_1} X_s^x \\
&\ = \ \nabla g(X_s^x) D_{V_2} \nabla_{u_1} X_s^x+\frac{s}{t} \nabla^2 g(X_s^x) \nabla_{u_2} X_s^x \nabla_{u_1} X_s^x
\end{split}
\Ee
with $D_{V_2} \nabla_{u_1} X_0^x=0$, where the second equality is by \eqref{de:DV=Js}.
For further use, we define
\
\be{
 \mcl I^{x}_{u_1}(t):\ =\ \frac{1}{\sqrt 2 t}\int_0^t \Ll \nabla_{u_1} X_s^x, \dif B_s\Rr,
}
\be{
 \mcl I^{x}_{{u_1},u_2}(t):\ = \ \mcl I_{u_1}^x(t) \mcl I_{u_2}^x(t)-D_{V_2} \mcl I_{u_1}^x(t).
}

\vskip 3mm

The following upper bounds on Malliavin derivatives are proven in Section 8.
\begin{lem} \label{l:GraDeX}
Let $u_i \in \R^d$ for $i=1,2$, and let
$$V_{i,s}\ = \ \int_0^{s} v_i(r) \dif r \ \ {\rm for} \ \ 0 \le s \le t,$$
where $v_i(r)=\frac{1}{\sqrt 2 t} \nabla_{u_i} X_r^x$ for $0 \le r \le t$. 
We have
\ \ \ \Be \label{e:DVNuXsp>}
\begin{split}
|D_{V_2} \nabla_{u_1} X_s^x|  \ \le \ C_{\theta} |u_1| |u_2|.
\end{split}
\Ee
\end{lem}

\begin{lem}  \label{l:LambEst}
Let $u_1,u_2 \in \R^d$ and $x \in \R^d$.
For all $p \ge 1$, $x \in \R^d$, we have
\Be \label{de:The1Est}
\E|\mcl I_{u_1}^x(t)|^{p} \ \le \ \frac{C_{\theta,p}  |u_1|^{p}}{t^{p/2}},
\Ee
\Be  \label{de:DThe1Est>}
\E\left|\nabla_{u_2} \mcl I_{u_1}^x(t)\right|^{p} \ \le \ \frac{C_{\theta,p}   |u_1|^{p} |u_2|^{p}}{t^{p/2}},
\Ee
\Be  \label{e:DVLamU>}
\begin{split}
\E |D_{V_2} \mcl I_{u_1}^x(t)|^{p} & \ \le \ \frac{C_{\theta,p}   |u_2|^{p} |{u_1}|^{p}}{t^{p}},
\end{split}
\Ee
\Be  \label{de:The2Est>}
\E|\mcl I_{u_1, u_2}^x(t)|^{p} \ \le \ \frac{C_{\theta,p}  |u_1|^{p}|u_2|^{p}}{t^{p}}.
\Ee
\end{lem}

\section{The representations of $f, \nabla f$ and $\nabla^2 f$} \label{s:Rep}
It is well known that SDE \eqref{de:OU} has the following infinitesimal generator $\mcl A$ \cite[Chapter VII]{ReYo99} defined by
\
\be{  
\mcl A f(x)\ = \ \Delta f+\Ll g(x), \nabla f(x)\Rr, \ \ \ \  \ \ f \in \mcl D(\mcl A),
}
where $\mcl D(\mcl A)$ is the domain of $\mcl A$, whose exact definition depends on the underlying function space that we consider.
$\mcl A$ generates a Markov semigroup $(P_t)_{t \ge 0}$ defined by
$$P_t f(x)\ = \ \E[f(X_t^x)], \ \ \ \ f \in \mcl C_b(\R^d,\R).$$
Note that $P_t:\mcl C_b(\R^d,\R)\rightarrow \mcl C_b(\R^d,\R)$ is a linear operator.
It is well known that $P_t$ can be extended to an operator $P_t: L^p_\mu(\R^d, \R) \rightarrow L^p_\mu(\R^d,\R)$ with $p \ge 1$, where $L^p_\mu(\R^d,\R)$ is the collection of all measurable functions $f:\R^d \rightarrow \R$ such that $\int_{\R^d} |f(x)|^p \mu(\dif x)<\infty$. Moreover, we have
\ \ \
\Be  \label{e:PtfRep}
P_t f(x)\ = \ \E[f(X_t^x)], \ \  \ \ \ \ \ \  f  \in L^p_\mu(\R^d,\R).
\Ee
\vskip 3mm

The Stein equation \eqref{de:StEq} can be written as
\Be  \label{e:PoiEqn}
\mcl A f(x)\ = \ h(x)-\mu(h),
\Ee
which is called the Poisson equation associated with $\mcl A$.
The solution is given as follows.
\ \ \ \
\begin{proposition}  \label{p:fRep}
For any Lipschitz function $h: \R^d \rightarrow\R$ with $||\nabla h||<\infty$, we have the following two statements:

\noindent (1). A solution to \eqref{de:StEq} is given by
\Be  \label{e:SolRep}
f(x)\ = \ -\int_0^\infty \E[h(X^x_t)-\mu(h)]\dif t.
\Ee
Moreover, we have
\Be \label{e:fBound}
|f(x)| \ \le \ C_{\theta_0, \theta_1, \theta_2, d}(1+|x|)||\nabla h||.
\Ee
(2). We have
\
\Be \label{de:ResRep}
f(x)\ = \ \int_0^\infty e^{-t} \E[f(X^x_t)+\mu(h)-h(X^x_t)] \dif t.
\Ee
\end{proposition}

\begin{rem}
The representation of $f$ in \eqref{de:ResRep} plays a crucial role in estimating $\nabla_{u_1} \nabla_{u_2} f$. We roughly explain it as follows. By a similar argument to that used to prove \eqref{e:g''4} below, we can show formally that
\ \ \ \
\Bes
\begin{split}
\nabla_{u_1}\nabla_{u_2} f(x)\ = \ \int_0^\infty \E\left[\nabla_{u_1} h(X_t^x) \mcl I_{u_2}^x(t)\right] \dif t.
\end{split}
\Ees
However, it is not known whether this integral is well defined.
Instead, we borrow the idea from
\cite[Section 4]{DaGo01} to introduce a new term $e^{-  t}$, and the corresponding new representation \eq{e:g''4} will produce an integrability.
\end{rem}

\begin{proof}
(1). Recall \eqref{e:PtfRep} and denote $\hat h=\mu(h)-h$. Let us first show that $\int_0^\infty P_s \hat h(x) \dif s$ is well defined. By \eqref{e:dWErg} in Appendix A,
we have
\Bes
\begin{split}
 \sup_{\|\nabla h\| \le 1} |P_t h(x)-\mu(h)|\ = \ \sup_{\|\nabla h\| \le 1} |\E[h(X_t^x)]-\mu(h)| \le 2 e^{-ct} \sup_{\|\nabla h\| \le 1} |h(x)-\mu(h)|,
\end{split}
\Ees
where $c$ depends on $\theta_0, \theta_1, \theta_2$.
Because
$$\sup_{\|\nabla h\| \le 1} |h(x)-\mu(h)| \ \le \ \int_{\R^d} |y-x| \mu(\dif y) \ \le \ (m_1(\mu)+|x|),$$
where $m_1(\mu)$ denotes the first absolute moment of $\mu$, we have
\ \ \
\Be  \label{e:ErbEst}
\begin{split}
|P_t \hat h(x)|\ = \ |P_t h(x)-\mu(h)| \ \le \ 2e^{-ct}  (m_1(\mu)+|x|) \|\nabla h\|, \ \ \ \ \ \forall  \ t>0.
\end{split}
\Ee
By \eqref{e:2ndMom}, $m_{1}(\mu)$ is finite; hence, $\left|\int_0^\infty P_t \hat h(x) \dif t\right|<\infty$.

For any $\e>0$, it is well known that $\e-\mcl A$ is invertible, and
 \ \ \
 \Bes
 (\e-\mcl A)^{-1} \hat h\ = \ \int_0^\infty e^{-\e t} P_t \hat h \ \dif t;
 \Ees
that is,
 \ \ \
 \Bes
 \e \int_0^\infty e^{-\e t} P_t \hat h \ \dif t-\hat h\ = \ \mcl A\left(\int_0^\infty e^{-\e t} P_t \hat h \ \dif t\right).
 \Ees
 As $\e \rightarrow 0+$,
$$\e \int_0^\infty e^{-\e t} P_t \hat h \ \dif t-\hat h \longrightarrow -\hat h,\ \ \ \ \ \ \int_0^\infty e^{-\e t} P_t \hat h \ \dif t \longrightarrow \int_0^\infty P_t \hat h \ \dif t.$$
As $\mcl A$ is
a closed operator, $\int_0^\infty P_t \hat h \ \dif t$ is in the domain of $\mcl A$ and
\Bes
-\hat h(x)\ = \ \mcl A \left(\int_0^\infty P_s \hat h(x) \dif s\right).
\Ees
Therefore, \eq{e:SolRep} is a solution to Eq. \eqref{de:StEq}.

By \eqref{e:ErbEst},
\
\Bes
\begin{split}
|f(x)| \ \le \ \left|\int_0^\infty P_t \hat h(x) \dif t\right| \ \le \ C_{\theta_0,\theta_1,\theta_2, d} (1+|x|)||\nabla h||.
\end{split}
\Ees
Hence, \eqref{e:fBound} is proven.

Now we prove (2). Note that
\
\Bes
(1-\mcl A) f(x)\ = \ f(x)+\hat h(x).
\Ees
By the integral representation of $(1-\mcl A)^{-1}$, we have
\
\Bes
f(x)\ = \ (1-\mcl A)^{-1}[ f+\hat h](x)\ = \ \int_0^\infty e^{-t} P_t (f+\hat h)(x) \dif t,
\Ees
which is \eqref{de:ResRep}.
\end{proof}



\begin{lem}  \label{l:BEL}
Let $\phi \in \mcl C^{1}(\R^{d},\R)$ be such that $\|\nabla \phi\| < \infty$, and let
$u, u_1, u_2 \in \R^d$. For every $t>0$ and $x \in \R^{d}$, we have
\ \ \
\Be  \label{e:GraPtPhi}
\left|\nabla_u \E\left[\phi(X^{x}_{t})\right]\right| \ \le \ \|\nabla \phi\| |u|,
\Ee
and
\
\Be \label{e:BELFor}
\nabla_u  \E\left[\phi(X^{x}_{t})\right]\ =\ \E\left[\phi(X_t^x)  \mcl I_u^x(t)\right].
\Ee
If, in addition, $\phi \in \mcl C^{2}(\R^{d},\R)$, then we have
\Be  \label{e:2ndDerEqn-1}
\nabla_{u_2} \nabla_{u_1} \E\left[\phi(X^{x}_{t})\right]\ = \ \E\left[\nabla_{u_1} \phi(X_t^x) \mcl I_{u_2}^x(t)\right],
\Ee
\Be  \label{e:2ndDerEqn-3}
\E[\nabla_{u_2} \phi(X_t^x) \mcl I_{u_1}^x(t)]\ =\ \E\left[\phi(X_t^x) \mcl I_{u_2,u_1}^x(t)\right].
\Ee
\end{lem}

\begin{proof}
 By Lebesgue's dominated convergence theorem and \eqref{de:Jac1Est>},
\ \ \
\be{
\begin{split}
\left|\nabla_u \E\left[\phi(X^{x}_{t})\right]\right|&\ =\ \left|\E[\nabla \phi(X_t^x)\nabla_u X_t^x]\right| \ \le \ \E[|\nabla \phi(X_t^x)| |\nabla_u X_t^x|]  \ \le \ \|\nabla \phi\| |u|.
\end{split}
}
 By \eqref{de:vs} and \eqref{de:DV=J}, we have
\be{
\begin{split}
\nabla_{u} \E\left[\phi(X^{x}_{t})\right]
&\ =\ \E\left[\nabla \phi(X_t^x) \nabla_{u} X_t^x \right] \\
&\ =\ \E\left[\nabla \phi(X_t^x) D_{V} X_t^x\right] \\
&\ =\ \E\left[D_{V} \phi(X_t^x)\right] \\
&\ =\ \E\left[\phi(X_t^x) \mcl I_{u}^x(t)\right],
\end{split}
}
where the last equality is by Bismut's formula \eqref{de:IBP}.
This proves \eq{e:BELFor}. From \eq{e:BELFor},
we have
\
\Bes
\begin{split}
\nabla_{u_2} \nabla_{u_1}  \E\left[\phi(X^{x}_{t})\right]\ =\ \nabla_{u_2} \E\left[\nabla_{u_1} \phi(X_t^x) \right]\ =\ \E\left[\nabla_{u_1} \phi(X_t^x) \mcl I_{u_2}^x(t)\right],
\end{split}
\Ees
where we used the assumption $\phi \in \mcl C^{2}(\R^{d},\R)$.
This proves \eq{e:2ndDerEqn-1}.
From \eqref{de:vs}, \eqref{de:DV=J}, \eqref{de:ProRul}, \eqref{de:IBP} and a similar calculation, we have
\be{
\begin{split}
&\ \ \ \ \ \ \E\left[\nabla_{u_2} \phi(X_t^x) \mcl I_{u_1}^x(t)\right] \\
&=\ \E\left[\nabla \phi(X_t^x) \nabla_{u_2} X_t^x \mcl I_{u_1}^x(t)\right]\\
&=\ \E\left[\nabla \phi(X_t^x) D_{V_2} X_t^x \mcl I_{u_1}^x(t)\right] \\
&=\ \E\left[D_{V_2} \phi(X_t^x) \mcl I_{u_1}^x(t)\right] \\
&=\ \E\left\{D_{V_2}\left[\phi(X_t^x) \mcl I_{u_1}^x(t)\right]\right\}-\E\left\{\phi(X_t^x)D_{V_2} \mcl I_{u_1}^x(t)\right\} \\
&=\ \E\left\{\left[\phi(X_t^x) \mcl I_{u_1}^x(t)\mcl I_{u_2}^x(t) \right]\right\}-\E\left\{\phi(X_t^x)D_{V_2} \mcl I_{u_1}^x(t)\right\}.
\end{split}
}
Thus, \eqref{e:2ndDerEqn-3} is proven.
\end{proof}
\begin{rem}
Write $P_{t} \phi(x)=\E[\phi(X^{x}_{t})]$, we can see that \eqref{e:BELFor} is the well known Bismut-Elworthy-Li formula \cite[(16)]{ElLi94}. The original proof of this formula is by It\^{o}'s formula and isometry \cite[p.254]{ElLi94}, while our approach is by \eqref{de:DV=J} and Bismut's integration by parts formula \eqref{de:IBP}. The idea in our proof has appeared in \cite{No86}, and been applied to study other problems such as the derivative formula of stochastic systems \cite{HaMa06,WXZ15}. Using Bismut's formula  two times, we obtain \eqref{e:2ndDerEqn-3}, which is crucial in proving \eqref{e:D2FEst2>}.
 Although \cite[(14)]{ElLi94} also gives a second order Bismut-Elworthy-Li formula, it is not directly applicable in our analysis because the first term on the right-hand side of (14) is not integrable at $0$.
\end{rem}

\ \ \ \
\begin{lem} \label{l:Nabf}
Let $h \in \mcl C^{1}(\R^{d},\R)$ be such that $\|\nabla h\| < \infty$. For any $u, u_1, u_2 \in \R^d$ and $x \in \R^d$, we have
\Be  \label{e:g'0}
\begin{split}
\nabla_u f(x)\ = \ \int_0^\infty \E\left[\nabla h(X_t^x) \nabla_u X_t^x\right] \dif t,
\end{split}
\Ee
\Be  \label{e:g'1}
\begin{split}
\nabla_u f(x)\ = \ \int_0^\infty e^{-t} \E\left\{\big[ f(X_t^x)-h(X_t^x)+\mu(h)\big]   \mcl I_u^x(t)\right\} \dif t.
\end{split}
\Ee
If in addition $h \in \mcl C^{2}(\R^{d},\R)$, then
\Be \label{e:g''4}
\begin{split}
\nabla_{u_2} \nabla_{u_1} f(x)&\ = \
\int_0^{\infty} e^{- t} \E\left\{\big[\nabla_{u_1} f(X_t^x)-\nabla_{u_1} h(X_t^x)\big] \mcl I_{u_2}^x(t)\right\} \dif t.
\end{split}
\Ee
\end{lem}

\begin{proof}
By \eqref{de:Jac1Est>}, we have
\ \ \ \
\be{  
\begin{split}
\left|\nabla_u \E\left[h(X_t^x)-\mu(h)\right]\right| & \ \le \  \E |\nabla h(X_t^x)||\nabla_u X_t^x| \\
&  \ \le \ \|\nabla h\| \E |\nabla_u X_t^x| \ \le \   \|\nabla h\| |u| e^{-\theta_0 t}, \ \ \ \ \ t>0.
\end{split}
}
Therefore, by the dominated convergence theorem, we have
\ \ \ \
\be{ 
\begin{split}
\nabla_u f(x)&\ = \ \int_0^\infty \nabla_u \E[h(X_t^x)-\mu(h)] \dif t \\
&\ = \ \int_0^\infty \E\left[\nabla h(X_t^x) \nabla_u X_t^x\right] \dif t.
\end{split}
}
The previous two relations also imply
\ \ \
\Be  \label{e:3.4-early}
\|\nabla_u f\|\ \ \le \ C_{\theta} \|\nabla h\| |u|.
\Ee
By \eqref{e:GraPtPhi} and \eqref{e:3.4-early}, we have
$$\|\nabla_u P_t [ f+\mu(h)-h]\| \ \le\  (\|\nabla f\|+\|\nabla h\|) |u|\ \le \ C_\theta \|\nabla h\| |u|.$$
By \eqref{de:ResRep}, the dominated convergence theorem and \eqref{e:BELFor}, we have
\
\be{  
\begin{split}
\nabla_u f(x)&\ = \ \int_0^\infty e^{- t} \nabla_u \E\big[ f(X_t^x)-h(X_t^x)+\mu(h)\big]  \dif t \\
&\ = \ \int_0^\infty e^{- t} \E\left\{\big[ f(X_t^x)-h(X_t^x)+\mu(h)\big]   \mcl I_u^x(t)\right\} \dif t.
\end{split}
}
This proves \eq{e:g'1}.
When $h \in \mcl C^{2}(\R^{d},\R)$, it can be checked that $f\in \mcl C^2(\R^d,\R)$ and
\be{
\nabla_{u_1} f(x)\ = \ \int_0^\infty e^{-t} \E[\nabla_{u_1}f(X_t^x)-\nabla_{u_1} h(X_t^x)]dt.
}
By the dominated convergence theorem with \eq{de:The1Est} and \eq{e:3.4-early}, and by \eqref{e:BELFor} with
$\phi=\nabla_{u_1}f$ and $\phi=\nabla_{u_1} h$, we have
\Bes
\begin{split}
\nabla_{u_2}\nabla_{u_1} f(x)
&\ = \ \int_0^\infty e^{- t}  \E\left\{\big[\nabla_{u_1} f(X_t^x)-\nabla_{u_1}h(X_t^x)\big] \mcl I_{u_2}^x(t)\right\} \dif t.
\end{split}
\Ees


\end{proof}
\ \ \ \ \

\section{The proof of Theorem \ref{t:D2FEst}}
Now let us use the representations of $f, \nabla f$ and $\nabla^{2}f$ developed in the previous section to prove Theorem \ref{t:D2FEst}.

\begin{lem}
Let $h \in \mcl C^{2}(\R^{d},\R)$ be such that $\|\nabla h\| < \infty$. Then we have
 \Be \label{e:3.5-1}
\begin{split}
\left|\nabla_{u_2}\nabla_{u_1} f(x)\right| \ \le \ C_\theta \|\nabla h\| |u_1| |u_2|,
\end{split}
\Ee
 \Be \label{e:3.6-1}
|\nabla_{u_2}\nabla_{u_1} f(x+\e u)-\nabla_{u_2}\nabla_{u_1} f(x)| \ \le \  C_{\theta}  ||\nabla h|| |\e| \left(|\log |\e|| \vee 1\right) |u_{1}||u_{2}|,
\Ee
for all $\e \in \R$, $x, u_{1}, u_{2} \in \R^d$ and $u \in \R^d$ with $|u| \le 1$.
\end{lem}
\begin{proof}
 From \eqref{e:g''4}, \eqref{e:3.4-early} and \eqref{de:The1Est}, we have
 \Bes
\begin{split}
\left|\nabla_{u_2}\nabla_{u_1} f(x)\right| & \ \le \ \int_0^\infty e^{- t}  \left|\E\left\{\big[\nabla_{u_1} f(X_t^x)-\nabla_{u_1}h(X_t^x)\big] \mcl I_{u_2}^x(t)\right\}\right| \dif t \\
& \ \le \ \left(\|\nabla f\|+\|\nabla h\|\right) |u_1| \int_0^\infty e^{- t}  \E\left[\left|\mcl I_{u_2}^x(t)\right|\right] \dif t \\
& \ \le \ C_\theta \|\nabla h\| |u_1| |u_2| \int_0^\infty e^{- t} t^{-1/2} \dif t \\
& \ \le \ C_\theta \|\nabla h\| |u_1| |u_2|.
\end{split}
\Ees
This proves \eq{e:3.5-1}. To prove \eq{e:3.6-1},
without loss of generality, we assume $\e>0$. By \eqref{e:g''4}, we have
\
\be{
\nabla_{u_2}\nabla_{u_1} f(x+\e u)-\nabla_{u_2}\nabla_{u_1}f(x)\ = \ \int_0^{\e^2}  e^{- t} \Psi \dif t+\int_{\e^2}^\infty  e^{- t} \Psi \dif t,
}
where
\be{
\begin{split}
&\Psi\ = \ \E\left\{\big[ \nabla_{u_1} f(X_t^{x+\e u})-\nabla_{u_1} h(X_t^{x+\e u})\big] \mcl I_{u_2}^{x+\e u}(t)\right\}\\
&\ \ \ \ \ \ \ \ -\E\left\{\big[ \nabla_{u_1} f(X_t^x)-\nabla_{u_1} h(X_t^x)\big] \mcl I_{u_2}^x(t)\right\}.  \\
\end{split}
}
We shall prove that
\ \ \ \
\Be  \label{e:1IntLog}
\left|\int_0^{\e^2}  e^{- t} \Psi \dif t\right| \ \le \ C_{\theta} \|\nabla h\| |u_1||u_2|  \e, \ \ \ \ \left|\int_{\e^2}^\infty  e^{- t} \Psi \dif t\right| \ \le \ C_\theta  (|\log |\e|| \vee 1) |\e| \|\nabla h\| |u_1| |u_2|.
\Ee
From these two inequalities, we immediately obtain \eqref{e:3.6-1}, as desired.

By \eqref{e:3.4-early} and \eqref{de:The1Est}, we have
\
\Bes
\begin{split}
\left|\Psi\right| & \ \le \ 2\left(\|\nabla f\|+\|\nabla h\|\right) |u_1|\E\big[\big|\mcl I^{x+\e u}_{u_2}(t)\big|+\big|\mcl I_{u_2}^x(t)\big|\big]\ \le \ C_{\theta} t^{-\frac 12} \|\nabla h\|  |u_1||u_2|,
\end{split}
\Ees
from which we obtain the first inequality in \eqref{e:1IntLog}.

We still need to prove the second inequality in \eqref{e:1IntLog}.
Note that
\be{
\Psi\ = \ \Psi_1+\Psi_{2}
}
where
\
\Bes
\begin{split}
&\Psi_{1}\ = \ \E\left\{\big[\nabla_{u_1} f(X_t^{x+\e u})-\nabla_{u_1} h(X_t^{x+\e u})\big]\big[\mcl I^{x+\e u}_{u_2}(t)-\mcl I_{u_2}^x(t)\big]\right\}, \\
&\Psi_{2}\ = \ \E\left\{\big[\nabla_{u_1} f(X_t^{x+\e u})-\nabla_{u_1} h(X_t^{x+\e u})-\nabla_{u_1} f(X_t^x)+\nabla_{u_1} h(X_t^x)\big] \mcl I_{u_2}^x(t)\right\}.
\end{split}
\Ees
For $\Psi_1$, we have
\ \ \
\Bes
\begin{split}
|\Psi_{1}|& \ \le \ \left(\|\nabla f\|+\|\nabla h\|\right) |u_1|\E\big[\left|\mcl I^{x+\e u}_{u_2}(t)-\mcl I_{u_2}^x(t)\right|\big] \\
& \ \le \ C_\theta \|\nabla h\|  |u_1| \E\left|\int_0^\e \left[\nabla_u \mcl I^{x+r u}_{u_2}(t)\right] \dif r\right| \\
& \ \le \ C_\theta \|\nabla h\|  |u_1| \int_0^\e \E\left|\nabla_u \mcl I^{x+r u}_{u_2}(t)\right| \dif r \\
& \ \le \ C_\theta \e \|\nabla h\|  |u_1| |u_2|   t^{-\frac 12},
\end{split}
\Ees
where the last inequality is by \eqref{de:DThe1Est>} and $|u|\le 1$.
For $\Psi_{2}$, by \eqref{e:2ndDerEqn-3}, we have
\ \ \ \
\Bes
\begin{split}
\left|\Psi_{2}\right|&\ = \ \left|\int_0^\e \E\left\{\nabla_u [\nabla_{u_1} f(X_t^{x+r u})-\nabla_{u_1} h(X_t^{x+r u})] \mcl I_{u_2}^x(t)\right\} \dif r\right| \\
&\ = \ \left|\int_0^\e  \E\left\{[\nabla_{u_1} f(X_t^{x+r u})-\nabla_{u_1} h(X_t^{x+r u})] \mcl I_{u,u_2}^x(t)\right\} \dif r\right| \\
&\ \le \ (\|\nabla f\|+\|\nabla h\|) |u_1| \int_0^\e  \E |\mcl I_{u,u_2}^x(t)| \dif r \\
& \ \le \ C_\theta \e \|\nabla h\| |u_1| |u_2| t^{-1},
\end{split}
\Ees
where the last inequality is by \eqref{de:The2Est>} and $|u| \le 1$.

Combining the estimates of $\Psi_1$ and $\Psi_2$, we obtain the second inequality in \eqref{e:1IntLog}.
\end{proof}

\begin{proof} [{\bf Proof of Theorem \ref{t:D2FEst}}]
Note that \eq{e:3.4-early} holds for any Lipschitz $h$,
which immediately implies \eqref{e:f'Est>}.

To prove the other two inequalities, it suffices to show that \eqref{e:3.5-1} and \eqref{e:3.6-1} hold for Lipschitz $h$. We now do so by a standard approximation.

Define
$$h_\delta(x)\ = \ \int_{\R^d} \phi_\delta(y) h(x-y) \dif y\ \ \ \  {\rm with} \ \ \delta>0,$$ where $\phi_\delta$ is the density function of the normal distribution
$N(0, \delta^2 I_d)$. It is easy to see that $h_\delta$ is smooth, $\lim_{\delta \rightarrow 0} h_\delta(x)=h(x)$, $\lim_{\delta \rightarrow 0} \nabla h_\delta(x)=\nabla h(x)$ and $|h_\delta(x)| \le C(1+|x|)$ for all $x \in \R^d$ and some $C>0$. Moreover, $\|\nabla h_\delta\| \le \|\nabla h\|$. The solution to the Stein equation \eqref{de:StEq}, with $h$ replaced by $h_\delta$, is
\be{  
f_\delta(x)\ =\ \int_0^\infty \E [h_\delta(X_t^x)-\mu(h_\delta)]\dif t.
}
Recall \eq{e:ErbEst}.
By the dominated convergence theorem,
$$\lim_{\delta \rightarrow 0} f_\delta(x)\ =\ \int_0^\infty \E [h(X_t^x)-\mu(h)]\dif t\ =\ f(x). $$
By \eqref{e:g'0} and the dominated convergence theorem,
\Bes
\lim_{\delta \rightarrow 0} \nabla_{u_1} f_\delta(x)\ =\ \lim_{\delta \rightarrow 0}\int_0^\infty \E\left[\nabla h_\delta(X_t^x) \nabla_{u_1} X_t^x\right]\dif t \ =\ \int_0^\infty \E\left[\nabla h(X_t^x) \nabla_{u_1} X_t^x\right]\dif t.
\Ees
As the differential operator $\nabla$ is closed \cite[Theorem 2.2.6]{Part04}, by the well known property of closed operators \cite[Proposition 2.1.4]{Part04}, we know that $f$ is differentiable and
\Bes
\nabla_{u_1} f(x) \ = \ \lim_{\delta \rightarrow 0} \nabla_{u_1} f_\delta(x).
\Ees
By \eqref{e:g''4}, we have
\Bes
\begin{split}
\nabla_{u_2} \nabla_{u_1} f_\delta(x)&\ = \
\int_0^{\infty} e^{- t} \E\left\{\big[\nabla_{u_1} f_\delta(X_t^x)+\nabla_{u_1} h_\delta (X_t^x)\big] \mcl I_{u_2}^x(t)\right\} \dif t,
\end{split}
\Ees
and by the dominated convergence theorem and the fact that $\nabla^2$ is closed, we have
\Be  \label{e:fetof-1}
\begin{split}
\lim_{\delta \rightarrow 0} \nabla_{u_2} \nabla_{u_1} f_\delta(x)&\ = \
\int_0^{\infty} e^{- t} \E\left\{\big[\nabla_{u_1} f(X_t^x)+\nabla_{u_1} h (X_t^x)\big] \mcl I_{u_2}^x(t)\right\} \dif t\ = \ \nabla_{u_2} \nabla_{u_1} f(x).
\end{split}
\Ee
By \eqref{e:3.5-1}, we have
 \Bes
\begin{split}
\left|\nabla_{u_2}\nabla_{u_1} f_\delta (x)\right| \ \le \ C_\theta \|\nabla h_\delta\| |u_1| |u_2|.
\end{split}
\Ees
Letting $\delta \rightarrow 0$, and by \eqref{e:fetof-1} and the fact that $\|\nabla h_\delta\| \le \|\nabla h\|$, we obtain \eqref{e:3.5-1} for Lipschitz $h$. Similarly we can prove \eqref{e:3.6-1} for Lipschitz $h$.
\end{proof}
\ \ \

\section{Proofs of the lemmas in Section \ref{s:SDE}}  \label{s:ProofSect5}
\begin{proof}[{\bf Proof of Lemma \ref{l:SolPro}}]
By It\^o's formula, \eq{a1a} and Cauchy's inequality, we have
\ \ \
\Bes
\begin{split}
\frac{d}{ds}\E |X_s^x|^2\  &=\ 2 \E \left[\Ll X_s^x, g(X_s^x)\Rr\right] +2d  \\
&\ = \ 2\E[\Ll X_s^x, g(X_s^x)-g(0)\Rr]+2\E[\Ll X_s^x, g(0) \Rr]+2d \\
& \ \le \ -2 \theta_0  \E  |X_s^x|^2 +\theta_0 \E|X_s^x|^2 +\frac{|g(0)|^2}{\theta_0} +2 d \\
& \ = \ - \theta_0  \E  |X_s^x|^2  +\frac{|g(0)|^2}{\theta_0} +2 d.
\end{split}
\Ees
This inequality, together with $X_0^x=x$, implies
\ \ \
\Bes
\begin{split}
\E |X_t^x|^2\  \le \ e^{-\theta_{0} t} |x|^2+(2d+\frac{|g(0)|^2}{\theta_0} ) \int_0^t  e^{-\theta_{0} (t-s)} \dif s \ \le \ e^{-\theta_{0} t} |x|^2+\frac{2d+|g(0)|^2/\theta_0}{\theta_{0}}.
\end{split}
\Ees
Let $\chi: [0,\infty) \rightarrow [0,1]$ be a continuous function such that $\chi(r)=1$ for $0 \le r \le 1$ and $\chi(r)=0$ for $r \ge 2$.
Let $R>0$ be a large number. The previous inequality implies
\Bes
\begin{split}
\E \left[|X_t^x|^2 \chi\left(|X_{t}^x|/R\right)\right]\ \le \ e^{-\theta_{0} t} |x|^2+\frac{2d+|g(0)|^2/\theta_0}{\theta_{0}}.
\end{split}
\Ees
Let $t \rightarrow \infty$. By the ergodicity of $X_{t}$ under weak topology (see Appendix A), we have
\Bes
\begin{split}
\int_{\R^{d}} |x|^2 \chi\left(|x|/R\right)\mu(\dif x) \  \le \ \frac{2d+|g(0)|^2/\theta_0}{\theta_{0}}.
\end{split}
\Ees
Letting $R \rightarrow \infty$, we obtain
\Bes
\begin{split}
\int_{\R^{d}} |x|^2 \mu(\dif x) \  \le \ \frac{2d+|g(0)|^2/\theta_0}{\theta_{0}}.
\end{split}
\Ees
\end{proof}

\


\begin{proof} [{\bf Proof of Lemma \ref{l:JacDev}}]
Recall $\theta_{0}>0$.
By \eqref{e:DuXt} and \eq{a2}, we have
\be{
\begin{split}
\frac{\dif}{\dif t} |\nabla_u X^x_t|^2& \ = \ 2 \Ll \nabla_u X^x_t, \nabla g(X^x_t) \nabla_u X^x_t \Rr \\
&\ \le \ -2\theta_0 \left(1+\theta_1 |X^x_t|^{\theta_2}\right)|\nabla_u X^x_t|^2,
\end{split}
}
which implies
\ \ \ \
\be{  
 |\nabla_u X^x_t|^2 \ \le \ \exp\left[-2\theta_0 \int_0^t \left(1+\theta_1 |X^x_s|^{\theta_2}\right) \dif s\right] |u|^2 \ \le \  e^{-2 \theta_0 t}|u|^2.
}
Writing $\zeta(t)=\nabla_{u_2}\nabla_{u_1}  X_t^x$, by \eqref{e:Du12Xt}, \eq{a2}, \eq{a3} and \eq{de:Jac1Est>},
we have
\ \ \
\Bes  
\begin{split}
\frac{\dif}{\dif t} |\zeta(t)|^2&\ =\ 2\Ll \zeta(t), \nabla g(X^x_t) \zeta(t)\Rr+2 \Ll \zeta(t), \nabla^2 g(X_t^x) \nabla_{u_2} X_t^x\nabla_{u_1} X_t^x\Rr \\
&\  \le \ -2\theta_0 \left(1+\theta_1 |X^x_t|^{\theta_2}\right)|\zeta(t)|^2+2\theta_3 \left(1+\theta_1 |X^x_t|\right)^{\theta_2-1}|\zeta(t)| |u_1| |u_2| \\
&\  \le \  -2\theta_0 \left(1+\theta_1 |X^x_t|^{\theta_2}\right)|\zeta(t)|^2+C_\theta  \left(1+\theta_1 |X^x_t|^{\theta_2}\right) |\zeta(t)| |u_1| |u_2| \\
&\ \le \ -\theta_0 \left(1+\theta_1 |X^x_t|^{\theta_2}\right)|\zeta(t)|^2+C_\theta  \left(1+\theta_1 |X^x_t|^{\theta_2}\right) |u_1|^2|u_2|^2,
\end{split}
\Ees
where the third inequality is by Cauchy's inequality. Recall that $\zeta(0)=0$. The above inequality implies
\ \ \
\Bes  
\begin{split}
|\zeta(t)|^2&\ \le \ C_\theta \int_0^t \exp\left[-\theta_0\int_s^t  \left(1+\theta_1 |X^x_r|^{\theta_2}\right) \dif r\right] \left(1+\theta_1 |X^x_s|^{\theta_2}\right) \dif s |u_1|^2|u_2|^2 \\
& \ \le \ C_\theta |u_1|^2 |u_2|^2,
\end{split}
\Ees
where the last inequality is by the following observation: for a nonnegative function $a: [0,t] \rightarrow \R$,
\be{  
\int_0^t e^{-\int_s^t a(r) \dif r} a(s) \dif s \ = \ 1-e^{-\int_0^t a(r) dr} \ \le \ 1.
}
Hence, \eqref{de:Jac2Est>} is proven.
\end{proof}

\begin{proof} [{\bf Proof of Lemma \ref{l:GraDeX}}]
From \eq{003} and
 the same argument as in the proof of Lemma \ref{l:JacDev}, we obtain the estimate in the lemma, as desired.  
\end{proof}

\begin{proof}[{\bf Proof of Lemma \ref{l:LambEst}}] Thanks to H\"{o}lder's inequality, it suffices to prove the inequalities for $p \ge 2$ in the lemma.
 By the Burkholder-Davis-Gundy inequality \cite[p. 160]{ReYo99} and \eqref{de:Jac1Est>}, we have
\
\be{
\begin{split}
\E|\mcl I_{u_1}^x(t)|^{p}& \ \leq \ \frac{C_{p}}{t^{p}}\E \left(\int_0^t |\nabla_{u_1} X_s^x|^2 \dif s\right)^{p/2} \ \le \  \frac{C_{p}|u_1|^{p}}{t^{p/2}}.
\end{split}
}
By the Burkholder-Davis-Gundy inequality and \eqref{de:Jac2Est>}, we have
\
\be{
\begin{split}
\E|\nabla_{u_2} \mcl I_{u_1}^x(t)|^{p}&\ \le \ \frac{C_{p}}{t^{p}}\E\left(\int_0^t |\nabla_{u_2} \nabla_{u_1}X_s^x|^{2} \dif s\right)^{p/2}  \ \le \ \frac{C_{\theta,p}|u_{1}|^{p} |u_{2}|^{p}}{t^{p/2}}. \\
\end{split}
}
It is easy to see that $D_{V_2} \mcl I_{u_1}^x(t)$ can be computed by \eqref{de:DVInt} as
\be{ 
\begin{split}
D_{V_2} \mcl I_{u_1}^x(t)\ =\ \frac{1}{\sqrt 2 t}\int_0^t \Ll D_{V_2} \nabla_{u_1} X_s^x,  \dif B_s\Rr+\frac{1}{2t^2}\int_0^t \Ll \nabla_{u_1} X_s^x, \nabla_{u_2} X_s^x\Rr \dif s.
\end{split}
}
By \eqref{e:DVNuXsp>}, Burkholder's and H\"{o}lder's inequalities, we have
\ \ \ \
\begin{align*}
\E \left|\frac{1}{\sqrt 2 t}\int_0^t \Ll D_{V_2} \nabla_{u_1} X_s^x,  \dif B_s\Rr\right|^{p} \le &\ \ \frac{C_p}{t^{p}}\E \left(\int_0^t |D_{V_2} \nabla_{u_1} X_s^x|^2 \dif s \right)^{p/2} \\
\le & \ \ \frac{C_p}{t^{p}} \left(\int_0^t \E |D_{V_2} \nabla_{u_1} X_s^x|^{p} \dif s\right) t^{\frac{p-2}2}   \\
 \le & \ \ \frac{C_{\theta,p} |u_2|^{p} |{u_1}|^{p}}{t^{p/2}}.
\end{align*}
By H\"{o}lder's inequality and \eqref{de:Jac1Est>}, we have
\ \ \
\Bes
\begin{split}
& \ \ \ \ \ \E \left|\frac{1}{2t^2}\int_0^t \Ll \nabla_{u_1} X_s^x, \nabla_{u_2} X_s^x\Rr \dif s\right|^{p} \\
 & \ \le \ \frac{1}{t^{2p}} \left(\int_0^t |\nabla_{u_1} X_s^x|^{p} |\nabla_{u_2} X_s^x|^{p} \dif s\right) t^{p-1} \ \le \ \frac{|{u_1}|^{p}|u_2|^{p}}{t^{p}}.
\end{split}
\Ees
Combining the previous three relations, we immediately obtain \eqref{e:DVLamU>}.

By the definition of $\mcl I_{u_1,u_2}^x(t)$, we have
\Bes
\begin{split}
\E|\mcl I_{u_1, u_2}^x(t)|^{p} & \ \le \ 2^{p-1} \E |\mcl I_{u_1}^x(t)\mcl I_{u_2}^x(t)|^{p}+2^{p-1} \E \left|D_{V_2} \mcl I_{u_1}^x(t) \right|^{p}.
\end{split}
\Ees
By \eqref{de:The1Est}, we have
\
\Bes
\begin{split}
\E |\mcl I_{u_1}^x(t)\mcl I_{u_2}^x(t)|^{p} & \ \le \ \sqrt{\E |\mcl I_{u_1}^x(t)|^{2p} \E |\mcl I_{u_2}^x(t)|^{2p}}  \\
& \ \le \ \sqrt{\E |\mcl I_{u_1}^x(t)|^{2p} \E |\mcl I_{u_2}^x(t)|^{2p}}\\
& \ \le \  \sqrt{\E |\mcl I_{u_1}^x(t)|^{2p} \E |\mcl I_{u_2}^x(t)|^{2p}}  \ \le \ C_{\theta, p} t^{-p}  |u_1|^{p} |u_2|^{p},
\end{split}
\Ees
which, together with \eqref{e:DVLamU>}, immediately gives \eqref{de:The2Est>}.
\end{proof}


  \appendix

  \section{On the ergodicity of SDE \eqref{de:OU}}

This section provides the details of the verification of the ergodicity of SDE \eqref{de:OU}.
 There are several ways to prove the ergodicity of SDE \eqref{de:OU}; here, we follow the approach used by Eberle \cite[Theorem 1 and Corollary 2]{Eb16} because it gives exponential convergence in Wasserstein distance. We verify the conditions in the theorem, adopting the notations in \cite{Eb16}. For any $r>0$, define
$$\kappa(r)\ = \ \inf\left\{-2 \frac{\Ll x-y, g(x)-g(y)\Rr}{|x-y|^2}: \ x,y \in \R^d \ s.t. \ |x-y|=\sqrt 2 r\right\}.$$
Compared with the conditions in \cite{Eb16}, SDE \eqref{de:OU} has $\sigma=\sqrt 2 I_d$ and $G=\frac 12 I_d$ and the associated intrinsic metric is $\frac{1}{\sqrt 2} |\cdot|$ with $|\cdot|$ being the Euclidean distance. By \eqref{a2}, we have
\ \ \ \
\Bes
\begin{split}
\Ll x-y, g(x)-g(y)\Rr&\ = \ \int_0^1 \Ll x-y, \nabla_{x-y}g(sx+(1-s)y)\Rr \dif s  \\
& \ \le \ -\theta_0\int_0^1 \left(1+\theta_1 |sx+(1-s)y|^{\theta_2}\right) |x-y|^2 \dif s,
\end{split}
\Ees
which implies that
$$\kappa(r) \ \ge \ \inf\left\{2\theta_0 \int_0^1 \left(1+\theta_1 |sx+(1-s)y|^{\theta_2}\right)  \dif s: \ x,y \in \R^d \ s.t. \ |x-y|=\sqrt 2 r\right\}.$$
Therefore, we have $\kappa(r)>0$ for $r>0$ and thus $\int_0^1 r \kappa(r)^{-} \dif r=0$.
Define
$$R_0\ = \ \inf\{R \ge 0: \ \kappa(r) \ge 0 \ \ \forall \ r \ge R\},$$
$$R_1\ = \ \inf\{R \ge R_0: \ \kappa(r) R(R-R_0)>8 \ \ \forall \ r \ge R\}.$$
It is easy to check that $R_0=0$ and $R_1 \in (0,\infty)$.

As
$\kappa(r)>0$ for all $r>0$, we have
$\varphi(r)=\exp\left(-\frac 14 \int_0^r s \kappa(s)^{-} \dif s\right)=1$ and thus $\Phi(r)=\int_0^r \varphi(s)\dif s=r$. Moreover, we have
$\alpha=1$ and
$$c\ = \ \left(\alpha \int_0^{R_1} \Phi(s) \varphi(s)^{-1} \dif s\right)^{-1}\ = \ \frac{2}{R^2_1}.$$
Applying Corollary 2 in \cite{Eb16}, we have
\ \ \
\Be \label{e:dWErg}
d_{\mcl W}(\mcl L(X_t^x),\mu) \ \le \ 2 e^{-ct} d_{\mcl W}(\delta_x,\mu), \ \ \ \ \forall \ t>0,
\Ee
where $\mcl L(X_t^x)$ denotes the probability distribution of $X_t^x$.
Note that the convergence rate $c>0$ only depends on $\theta_0, \theta_1$ and $\theta_2$.

From \eqref{e:dWErg}, we say that  $\mcl L(X_t^x) \rightarrow \mu$ weakly, in the sense that for any bounded continuous function
$f: \R^d \rightarrow \R$, we have
\Bes
\lim_{t\rightarrow \infty} \E f(X_t^x)\ = \ \mu(f),
\Ees
which immediately implies
\be{ 
\lim_{t\rightarrow \infty} \frac1t \int_0^t \E f(X_s^x) \dif s\ = \ \mu(f).
}




\section{Multivariate normal approximation}

In this appendix, we prove the results stated in Remark \ref{r8} with regard to multivariate normal approximation for sums of independent, bounded random vectors.

\begin{theorem}\label{t2}
Let $W=\frac{1}{\sqrt{n}}\sum_{i=1}^n X_i$ where  $X_1,\dots, X_n$ are independent $d$-dimensional random vectors such that $\E X_i=0$, $|X_i|\leq \beta$ and $\E W W^{\rm T}=I_d$. Then we have
\ben{\label{t2-1}
d_{\mcl W}(\mathcal{L}(W), \mathcal{L}(Z))\ \le  \frac{C d \beta}{\sqrt{n}}(1+\log n)
}
and
\ben{\label{t2-2}
d_{\mcl W}(\mathcal{L}(W), \mathcal{L}(Z))\ \le \  \frac{Cd^2 \beta}{\sqrt{n}},
}
where $C$ is an absolute constant and $Z$ has the standard $d$-dimensional normal distribution.
\end{theorem}



\begin{proof}
Note that by the same smoothing and limiting arguments as in the proof of Theorem \ref{t:MThm}, we only need to consider test functions $h\in \text{Lip}(\mathbb{R}^d, 1)$, which are smooth and have bounded derivatives of all orders. This is assumed throughout the proof so that the integration, differentiation, and their interchange are legitimate.

For multivariate normal approximation, the Stein equation \eq{de:StEq} simplifies to
\ben{\label{1}
\Delta f(w)-\Ll w, \nabla f(w) \Rr \ = \ h(w)-\E h(Z) .
}
An appropriate solution to \eq{1} is known to be
\ben{\label{2}
f_h(x)\ = \ -\int_0^{\infty} \{h*\phi_{\sqrt{1-e^{-2s}}}(e^{-s }x) -\E h(Z)\}ds,
}
where $*$ denotes the convolution and $\phi_r(x)=(2\pi r^2)^{-d/2}e^{-|x|^2/2r^2}$.
From \eq{1}, we have
\be{
d_{\mcl W}(\mathcal{L}(W), \mathcal{L}(Z))\ = \ \sup_{h\in \text{Lip}(\mathbb{R}^d, 1)}|\E W\cdot \nabla f(W)- \E \Delta f(W)|
}
with $f:=f_h$ in \eq{2} (we omit the dependence of $f$ on $h$ for notational convenience).

Let $C$ be a constant that may differ in different expressions.
Denote
\ben{\label{005}
\eta:\ = \ d_{\mcl W}(\mathcal{L}(W), \mathcal{L}(Z)).
}
Let $\{X_1',\dots, X_n'\}$ be an independent copy of $\{X_1,\dots, X_n\}$.
Let $I$ be uniformly chosen from $\{1,\dots, n\}$ and be independent of $\{X_1,\dots, X_n, X_1',\dots, X_n'\}$.
Define
$$W'\ = \ W-\frac{X_I}{\sqrt{n}}+\frac{X_I'}{\sqrt{n}}.$$
Then $W$ and $W'$ have the same distribution. Let
$$\delta\ = \ W'-W\ = \ \frac{X_I'}{\sqrt{n}}-\frac{X_I}{\sqrt{n}}.$$
We have, by the independence assumption and the facts that $\E X_i=0$ and $\E W W^{\rm T}=I_d$,
$$\E[\delta|W]\ = \ \frac{1}{n}\sum_{i=1}^n \E[\frac{X_i'}{\sqrt{n}}-\frac{X_i}{\sqrt{n}}|W]\ = \ -\frac{1}{n}W$$
and
$$\E[\delta\delta^{\rm T}|W]\ = \ \frac{2}{n}I_d +\frac{1}{n}\{\E[\frac{1}{n}\sum_{i=1}^n X_i X_i^{\rm T}|W]-I_d\}.$$
Therefore, \eq{e:A1} and \eq{e:A2} are satisfied with
$$\lambda\ = \ \frac{1}{n},\quad g(W)=-W, \quad R_1=0,\quad
R_2\ = \ \frac{1}{2}\{\E[\frac{1}{n}\sum_{i=1}^n X_i X_i^{\rm T}|W]-I_d\} \ $$
Note that this is Example 1 below Assumption \ref{a:A} with $\lambda_1=\dots=\lambda_d=1$.
By the boundedness condition, 
$$|\delta|\leq \frac{2\beta}{\sqrt{n}}.$$
We also have
$$\E[|\delta|^2]=\frac{2}{n}\sum_{i=1}^n \E[\frac{|X_i|^2}{n}]=\frac{2d}{n}.$$
As $\beta^2\geq \sum_{i=1}^n \E[\frac{|X_i|^2}{n}]=d$, we have $\beta\geq \sqrt{d}$.
Using these facts and assuming that $\beta\leq \sqrt{n}$ (otherwise \eq{t2-1} is trivial), 
in applying \eq{e:EstDel-0}, we have
$$\E|R_1|\ = \ 0,$$
\besn{\label{004}
&\sqrt{d}\E[||R_2||_{\rm HS}] \leq C\sqrt{d}\sqrt{\sum_{j,k=1}^d {\rm Var} [\frac{1}{n}\sum_{i=1}^n X_{ij} X_{ik}]}\\
= & C\sqrt{d}\sqrt{\sum_{j,k=1}^d\frac{1}{n^2}\sum_{i=1}^n {\rm Var} [ X_{ij} X_{ik}]}\leq C\sqrt{d}\sqrt{\sum_{j,k=1}^d\frac{1}{n^2}\sum_{i=1}^n \E [ X_{ij}^2 X_{ik}^2]}\\
=& C\sqrt{d}\sqrt{\frac{1}{n^2}\sum_{i=1}^n \E [ |X_{i}|^4]}\leq C\sqrt{d}\beta \sqrt{\frac{1}{n^2}\sum_{i=1}^n \E [ |X_{i}|^2]}=\frac{Cd\beta}{\sqrt{n}},
}
and
\bes{
&\frac{1}{\lambda}\E\left[|\delta|^3 \left(|\log|\delta|| \vee 1\right)\right]\\
\leq & Cn \frac{\beta}{\sqrt{n}} (1+\log n) \E[|\delta^2|]\leq \frac{Cd\beta}{\sqrt{n}}(1+\log n).
}
This proves \eq{t2-1}.

To prove \eq{t2-2}, we modify the argument from \eq{101} by using the explicit expression of $f$ in \eq{2}.
With $\delta_i=(X_i'-X_i)/\sqrt{n}$ and $h_s(x):=h(e^{-s}x)$, we have
\bes{
&\frac{1}{\lambda} \int_0^1 |\E[\langle \delta \delta^{\rm T},
\nabla^2 f(W+t\delta)-\nabla^2 f(W) \rangle_{\text{HS}} ] |dt\\
=&\int_0^1 |\E \sum_{i=1}^n \int_0^\infty \int_0^1 t \nabla_{\delta_i}^3 (h_s*\phi_{\sqrt{e^{2s}-1}})(W+ut\delta_i) du ds |   dt,
}
where $\nabla_{\delta_i}^3:=\nabla_{\delta_i}\nabla_{\delta_i}\nabla_{\delta_i}$ (cf. Section 2.1).
We separate the integration over $s$ into $\int_0^{\epsilon^2}$ and $\int_{\epsilon^2}^\infty$ with an $\epsilon$ to be chosen later.
For the part $\int_0^{\epsilon^2}$, exactly following \cite[pp 18-19]{VaVa10},
we have
\ben{\label{7}
C \E\sum_{i=1}^n \int_0^{\epsilon^2} e^{-s} |\delta_i|^2 \frac{1}{\sqrt{e^{2s}-1}} ds \ \le \  C d \epsilon,
}
where we used $\sum_{i=1}^n \E|\delta_i|^2=2d$.
The part $\int_{\epsilon^2}^\infty$ is treated differently.
Using the interchangeability of convolution and differentiation, we have
\besn{\label{6}
&\left| \E\sum_{i=1}^n \int_{\epsilon^2}^\infty \int_0^1 \nabla^3_{\delta_i} (h_s * \phi_{\sqrt{e^{2s}-1}})(W+ut\delta_i)duds \right|\\
\ = \ &\left| \E \sum_{i=1}^n \int_{\epsilon^2}^\infty \int_0^1 \int_{\mathbb{R}^d} h_s(W+ut\delta_i-x) \nabla^3_{\delta_i} \phi_{\sqrt{e^{2s}-1}} (x) dx du ds \right|.
}
Let $\{\hat{X}_1,\dots, \hat{X}_n\}$ be another independent copy of $\{X_1,\dots, X_n\}$ and be independent of $\{X_1',\dots, X_n'\}$, and let $\hat{W}_i=W-\frac{X_i}{\sqrt n}+\frac{\hat{X}_i}{\sqrt n}$.
From this construct, for each $i$, $\hat{W}_i$ has the same distribution as $W$ and is independent of $\{X_i, X_i'\}$.
Let $\hat{Z}$ be an independent standard Gaussian vector. Rewriting
\bes{
h_s(W+ut\delta_i-x)\ = \ &[h_s(W+ut\delta_i-x)-h_s(\hat{W}_i-x)]\\
&+[h_s(\hat{W}_i-x)-h_s(\hat{Z}-x)]\\
&+[h_s(\hat{Z}-x)],
}
the term inside the absolute value in \eq{6} is separated into three terms as follows:
\be{
R_{31}\ = \ \E \sum_{i=1}^n \int_{\epsilon^2}^\infty \int_0^1 \int_{\mathbb{R}^d} [h_s(W+ut\delta_i-x)-h_s(\hat{W}_i-x)] \nabla^3_{\delta_i} \phi_{\sqrt{e^{2s}-1}} (x) dx du ds,
}
\be{
R_{32}\ = \ \E \sum_{i=1}^n \int_{\epsilon^2}^\infty \int_0^1 \int_{\mathbb{R}^d} [h_s(\hat{W}_i-x)-h_s(\hat{Z}-x)] \nabla^3_{\delta_i} \phi_{\sqrt{e^{2s}-1}} (x) dx du ds,
}
\be{
R_{33}\ = \ \E \sum_{i=1}^n \int_{\epsilon^2}^\infty \int_0^1 \int_{\mathbb{R}^d} h_s(\hat{Z}-x) \nabla^3_{\delta_i} \phi_{\sqrt{e^{2s}-1}} (x) dx du ds.
}
We bound their absolute values separately.
From $h\in \text{Lip}(\mathbb{R}^d, 1)$, $h_s(x)=h(e^{-s}x)$ and $|X_i|\leq \beta$, we have
\be{
[h_s(W+ut\delta_i-x)-h_s(\hat{W}_i-x)]\ \le \  \frac{Ce^{-s}\beta}{\sqrt{n}}.
}
Moreover,
\be{
\int_{\mathbb{R}^d} |\nabla^3_{\delta_i} \phi_{\sqrt{e^{2s}-1}} (x) |dx \ \le \  C|\delta_i|^3\frac{1}{(e^{2s}-1)^{3/2}}.
}
Therefore,
\be{
|R_{31}|\ \le \  C\sum_{i=1}^n \E \int_{\epsilon^2}^\infty e^{-s} \frac{\beta |\delta_i|^3}{n^2} \frac{1}{(e^{2s}-1)^{3/2}}ds\ \le \  C \frac{d\beta^2}{\epsilon n},
}
where we used $|\delta_i|\leq 2\beta/\sqrt{n}$ and $\sum_{i=1}^n\E|\delta_i|^2=2d$.

From the definition of $\eta$ in \eq{005} and the fact that $\hat{W}_i$ has the same distribution as $W$, we have
\be{
|\E[h_s(\hat{W}_i-x)-h_s(\hat{Z}-x)]|\ \le \  e^{-s} \eta.
}
Using independence and the same argument as for $R_{31}$, we have
\be{
|R_{32}|\ \le \  \frac{Cd\beta \eta}{\epsilon\sqrt{n}}.
}
Now we bound $R_{33}$. Using integration by parts, and combining two independent Gaussians into one, we have
\bes{
|R_{33}|\ = \ &\left| \E \sum_{i=1}^n \int_{\epsilon^2}^\infty \int_0^1 \int_{\mathbb{R}^d} h_s(\hat{Z}-x) \nabla^3_{\delta_i} \phi_{\sqrt{e^{2s}-1}} (x) dx dt ds \right|\\
\ = \ &\left| \E \sum_{i=1}^n \int_{\epsilon^2}^\infty  \int_{\mathbb{R}^d} \nabla^3_{\delta_i} h_s(\hat{Z}-x)  \phi_{\sqrt{e^{2s}-1}} (x) dx  ds \right|\\
\ = \ &\left| \E \sum_{i=1}^n \int_{\epsilon^2}^\infty  \int_{\mathbb{R}^d} \nabla^3_{\delta_i} h_s(x)  \phi_{\sqrt{e^{2s}}} (x) dx  ds \right|\\
\ = \ &\left| \E \sum_{i=1}^n \int_{\epsilon^2}^\infty  \int_{\mathbb{R}^d}  \nabla_{\delta_i} h_s(x) D^2_{\delta_i}  \phi_{\sqrt{e^{2s}}} (x) dx  ds \right|\\
\ \leq\ &\left| \E \sum_{i=1}^n \int_{\epsilon^2}^\infty    |\delta_i| e^{-s} |\delta_i|^2 \frac{C}{e^{2s}}   ds \right| \ \le \  \frac{Cd\beta}{\sqrt{n}}.
}
From \eq{004}, \eq{7} and the bounds on $R_{31}, R_{32}$ and $R_{33}$, we have
\be{
\eta\ \le \  C(\frac{d \beta}{\sqrt{n}}+d\epsilon+\frac{d \beta^2}{\epsilon n}+\frac{d\beta \eta}{\epsilon \sqrt{n}}+\frac{d\beta}{\sqrt{n}}).
}
The theorem is proved by choosing $\epsilon$ to be a large multiple of $d\beta/\sqrt{n}$ and solving the recursive inequality for $\eta$.

\end{proof}

\begin{center}
{\bf Acknowledgements}
\end{center}

We thank Michel Ledoux for very helpful discussions.
We also thank two anonymous referees for their valuable comments which have improved the manuscript considerably.
Fang X. was partially supported by Hong Kong RGC ECS 24301617, a CUHK direct grant and a CUHK start-up grant.
Shao Q. M. was partially supported by Hong Kong RGC GRF 14302515 and 14304917.
Xu L. was partially supported by Macao S.A.R. (FDCT 038/2017/A1, FDCT 030/2016/A1, FDCT 025/2016/A1), NNSFC 11571390, University
of Macau (MYRG 2016-00025-FST, MYRG 2018-00133-FST).



\bibliographystyle{amsplain}

\begin{thebibliography}{99}




\bibitem{ABR99} Albeverio, S., Bogachev, V. and R\"ockner, M. (1999). On uniqueness of invariant measures for finite-and infinite-dimensional diffusions. \emph{Comm. Pure Appl. Math.} {\bf 52}, 325--362. \\

\bibitem{Be05}
Bentkus, V. (2005).
A Lyapunov type bound in ${\bf R}\sp d$.
\textit{Theory Probab. Appl.} \textbf{49}, 311--323. \\

\bibitem{Bob18}
 Bobkov, S. G. (2018) Berry-Esseen bounds and Edgeworth expansions in the central limit theorem for transport distances. {\em Probab. Theory Related Fields} {\bf 170}, no. 1-2, 229--262. \\

\bibitem{Bo18}
Bonis, T. (2018).
Rates in the central limit theorem and diffusion approximation via Stein's method.
\emph{Preprint}.
Available at \url{https://arxiv.org/abs/1506.06966} \\

\bibitem{BrDa17} Braverman, A. and Dai, J.G. (2017). Stein's method for steady-state diffusion approximations of $M/Ph/n+M$ systems. {\em Ann. Appl. Probab.} {\bf 27}, 550--581.\\

\bibitem{Sa01} Cerrai S. (2001). \emph{Second order PDE's in finite and infinite dimensions: A probabilistic approach.} x+330 pp. Lecture Notes in Mathematics {\bf1762}, Springer Verlag. \\

\bibitem{ChMe08}
Chatterjee, S. and Meckes, E. (2008). Multivariate normal approximation using exchangeable pairs. \textit{ALEA Lat. Am. J. Probab. Math. Stat.} {\bf 4}, 257--283. \\

\bibitem{ChSh11} Chatterjee, S. and Shao, Q.M. (2011). Nonnormal approximation by Stein's method of exchangeable pairs with application to the Curie-Weiss model. {\em Ann. Appl. Probab.} {\bf 21}, 464--483. \\

\bibitem{CoFaPa17}
Courtade, T.A., Fathi, M. and Pananjady, A. (2017).
Existence of Stein kernels under a spectral gap, and discrepancy bound.
{\em Preprint}. Available at \url{https://arxiv.org/abs/1703.07707}  \\

\bibitem{Da17} Dalalyan, A.S. (2017). Theoretical guarantees for approximate sampling from smooth and log-concave densities. {\em J. R. Stat. Soc. Ser. B. Stat. Methodol.} {\bf 79}, 651--676. \\

\bibitem{DaGo01} Da Prato, G. and Goldys, B. (2001). Elliptic operators on ${\bf R}\sp d$ with unbounded coefficients. {\it J. Differential Equations} {\bf 172}, 333--358. \\



\bibitem{DWT95} Down, D., Meyn S.P. and Tweedie, R.L. (1995). Exponential and uniform ergodicity of Markov processes. \emph{Ann. Probab.} {\bf 23}, 1671--1691.\\

\bibitem{DuMo18}
Durmus, A. and Moulines, E. (2018). High-dimensional Bayesian inference via the Unadjusted Langevin Algorithm. {\em Preprint}. Available at \url{https://arxiv.org/abs/1605.01559} \\

\bibitem{Eb16} Eberle, A. (2016). Reflection couplings and contraction rates for diffusions. \emph{Probab. Theory Related Fields} {\bf 166}, 851--886. \\

\bibitem{ElLi94} Elworthy, L. and Li, X.M. (1994). Formulae for the derivatives of heat semigroups. J. Funct. Anal. {\bf 125},  252--286. \\

\bibitem{GoRi96}
Goldstein, L. and Rinott, Y. (1996).
Multivariate normal approximation by Stein's method and size bias couplings. \textit{Appl. Probab. Index} \textbf{33}, 1--17. \\

\bibitem{Go17}
Gorham, J., Duncan, A.B., Vollmer, S.J. and Mackey, L. (2017).
Measuring sample quality with diffusions.
{\em Preprint}.
Available at \url{https://arxiv.org/abs/1611.06972} \\

\bibitem{Go91}
G\"otze, F. (1991).
On the rate of convergence in the multivariate CLT.
\textit{Ann. Probab.} \textbf{19}, 724--739. \\

\bibitem{Gu14}
Gurvich, I. (2014). Diffusion models and steady-state approximations for exponentially ergodic Markovian queues. {\em Ann. Appl. Probab.} {\bf 24}, 2527--2559. \\

\bibitem{HaMa06}
Hairer, M. and Mattingly, J.C. (2006). Ergodicity of the 2D Navier-Stokes equations with degenerate stochastic forcing. \emph{Ann. of Math.} {\bf 164}, 993--1032. \\

\bibitem{KuTu18}
Kusuoka, S. and Tudor, C.A. (2018). Characterization of the convergence in total variation and extension of the fourth moment theorem to invariant measures of diffusions. {\em Bernoulli} {\bf 24}, 1463--1496. \\

\bibitem{LNP15}
Ledoux, M., Nourdin, I. and Peccati, G. (2015). Stein's method, logarithmic Sobolev and transport inequalities. {\em Geom. Funct. Anal.} {\bf 25}, 256--306. \\

\bibitem{Les18} Mackey, L. \emph{Private communications}, Feb 2018. \\

\bibitem{MaGo16}
Mackey, L. and Gorham, J. (2016). Multivariate Stein factors for a class of strongly log-concave distributions. {\em Electron. Commun. Probab.} {\bf 21}, 1--14. \\

\bibitem{No86} Norris, J. (1986).  \emph{Simplified Malliavin calculus}. S\'{e}minaire de Probabilit\'{e}s, XX, 1984/85, 101-130, Lecture Notes in Math., {\bf 1204}, Springer, Berlin. \\

\bibitem{NoPe09}
Nourdin, I. and Peccati, G. (2009). Stein's method on Wiener chaos, {\em Probab. Theory Related Fields} {\bf 145}, 75--118. \\

\bibitem{NoPe12}
Nourdin, I. and Peccati, G. (2012). {\em Normal approximations with Malliavin calculus. From Stein's method to universality.} Cambridge Tracts in Mathematics {\bf 192}, Cambridge University Press, Cambridge. xiv+239 pp. \\


\bibitem{NoPeRe10}
Nourdin, I., Peccati, G. and R\'eveillac. A. (2010). Multivariate normal approximation using Stein's method and Malliavin calculus. {\em  Ann. Inst. Henri Poincare Probab. Stat.} {\bf 46}, 45--58. \\






\bibitem{Part04}
Partington, J.R. (2004). \emph{Linear operators and linear systems}. London Mathematical Society Student Texts {\bf 60}, Cambridge University Press.  \\

\bibitem{ReRo09}
Reinert, G. and R\"ollin, A. (2009).
Multivariate normal approximation with Stein's method of exchangeable pairs under a general linearity condition. \textit{Ann. Probab.} \textbf{37}, 2150--2173. \\

\bibitem{ReYo99}
Revuz, D. and Yor, M. (1999). \emph{Continuous martingales and Brownian motion}. 3rd edition. Grundlehren der Mathematischen Wissenschaften [Fundamental Principles of Mathematical Sciences] {\bf 293}, Springer-Verlag, Berlin. \\

\bibitem{RiRo96}
Rinott, Y. and Rotar, V. (1996). A multivariate CLT for local dependence with $n^{-1/2} \log n$ rate and applications to multivariate graph related statistics. \textit{J. Multivariate Anal.} \textbf{56}, 333--350. \\


\bibitem{RoTw96}
Roberts, G.O. and Tweedie, R.L. (1996).
Exponential convergence of Langevin distributions and their discrete approximations. \textit{Bernoulli} {\bf 4}, 341--363. \\

\bibitem{Ro08} R\"ollin, A. (2008). A note on the exchangeability condition in Stein's method. {\em Statist. Probab. Lett.} {\bf 78}, 1800--1806. \\

\bibitem{Sak85} Sakhanenko, A. I. (1985) {\em Estimates in an invariance principle}. (Russian) Limit theorems of probability theory, 27--44, 175, Trudy Inst. Mat., 5, 'Nauka' Sibirsk. Otdel., Novosibirsk. \\

\bibitem{ShZh16}
Shao, Q.M. and Zhang, Z.S. (2016). Identifying the limiting distribution by a general approach of Stein's method. 
{\it Sci. China Math.} {\bf  59}, 2379 -- 2392. \\



\bibitem{St72}
Stein, C. (1972).
A bound for the error in the normal approximation to the distribution of a sum of dependent random variables.
\textit{Proc. Sixth Berkeley Symp. Math. Stat. Prob.}
\textbf{2}, Univ. California Press. Berkeley, Calif., 583--602. \\


\bibitem{St86}
Stein, C. (1986).
\textit{Approximate Computation of Expectations.}
Lecture Notes 7, Inst. Math. Statist., Hayward, Calif. \\


\bibitem{VaVa10}
Valiant, G. and Valiant, P. (2010). A CLT and tight lower bounds for estimating entropy. {\em Electronic Colloquium on Computational Complexity.} TR10-179.\\

\bibitem{WXZ15}
Wang, F.Y., Xu, L. and Zhang, X. (2015).  Gradient estimates for SDEs driven by multiplicative Levy noise. \emph{J. Funct. Anal.} {\bf 269}, 3195--3219. \\

\bibitem{Zh17} Zhai, A. (2018).
A high-dimensional CLT in $\mathcal{W}_2$ distance with near optimal convergence rate. {\em Probab. Theory Related Fields} {\bf 170}, 821--845. \\

\end{thebibliography}

\end{document}